\documentclass[a4paper,oneside]{amsart}

\def\Wdn#1\wdn{\marginpar{\tiny #1}}
\long\def\WDN#1\wdn{[WDN: #1]\Wdn[Comment]\wdn}

\usepackage{amssymb}
\usepackage{amsmath}
\usepackage{amsfonts}
\usepackage{amsthm}
\usepackage{mathrsfs}
\usepackage{fix-cm}
\usepackage{graphicx}
\usepackage{amscd}

\usepackage{soul}
\sodef\spred{}{.2em}{.9em plus.4em}{1em plus.1em minus.1em}

\usepackage[latin1]{inputenc}
\usepackage[T1]{fontenc}

\usepackage[english]{babel}

\usepackage{bbm}

\usepackage[all]{xy}

\newbox\mybox
\def\overtag#1#2#3{\setbox\mybox\hbox{$#1$}\hbox to
  0pt{\vbox to 0pt{\vglue-#3\vglue-\ht\mybox\hbox to \wd\mybox
      {\hss$\ss#2$\hss}\vss}\hss}\box\mybox}
\def\undertag#1#2#3{\setbox\mybox\hbox{$#1$}\hbox to 0pt{\vbox to
    0pt{\vglue#3\vglue\ht\mybox\hbox to \wd\mybox
      {\hss$\ss#2$\hss}\vss}\hss}\box\mybox}
\def\lefttag#1#2#3{\hbox to 0pt{\vbox to 0pt{\vss\hbox to
      0pt{\hss$\ss#2$\hskip#3}\vss}}#1}
\def\righttag#1#2#3{\hbox to 0pt{\vbox to 0pt{\vss\hbox to
      0pt{\hskip#3$\ss#2$\hss}\vss}}#1}
\let\ss\scriptstyle

\def\Dot{\lower.2pc\hbox to 2.5pt{\hss$\bullet$\hss}}
\def\Circ{\lower.2pc\hbox to 2.5pt{\hss$\circ$\hss}}
\def\Vdots{\raise5pt\hbox{$\vdots$}}
\def\splicediag#1#2{\xymatrix@R=#1pt@C=#2pt@M=0pt@W=0pt@H=0pt}

\newcommand\lineto{\ar@{-}}
\newcommand\dashto{\ar@{--}}
\newcommand\dotto{\ar@{.}}

\newtheorem{thm}{Theorem}[section]
\newtheorem{lemma}[thm]{Lemma}
\newtheorem{prop}[thm]{Proposition}
\newtheorem{cor}[thm]{Corollary}

\newtheorem{thm*}{Theorem}

\theoremstyle{definition} 
\newtheorem{defn}[thm]{Definition} 

\newcommand{\Q}{\mathbbm{Q}}

\newcommand{\Z}{\mathbbm{Z}}

\newcommand{\num}[1]{\lvert #1 \rvert}

\newcommand{\lk}{\operatorname{lk}}

\newcommand{\im}{\operatorname{Im}}

\newcommand{\sign}{\operatorname{sign}}
\newcommand{\lcm}{\operatorname{lcm}}

\newcommand{\morf}[4][\to]{ #2 \colon #3 #1 #4}

\newcommand{\inv}{^{-1}}

\renewcommand{\phi}{\varphi}
\renewcommand{\epsilon}{\varepsilon}

\begin{document}

\bibliographystyle{alpha}

\nocite{*}


\title[Splice diagram, singularity links and universal
abelian covers]
{Splice diagram determining singularity links and universal
abelian covers}
\author{Helge M\o{}ller Pedersen}
\address{Department of Mathematics\\ Columbia University
\\ New York, NY 10027}
\email{helge@math.columbia.edu}
\begin{abstract}
  To a rational homology sphere graph manifold one can associate a
  weighted tree invariant called splice diagram. In this article we
  prove a sufficient numerical condition on the splice diagram for a
  graph manifold to be a singularity link. We also show that if two
  manifolds have the same splice diagram, then their universal abelian
  covers are homeomorphic. To prove the last theorem we have to
  generalize our notions to orbifolds.
\end{abstract}
\maketitle

\section{Introduction}

For prime $3$-manifolds $M$ one has several decomposition theorems, like
the geometric decomposition which cuts $M$ along embedded tori and
Klein bottles into geometric pieces, or the JSJ decomposition which
cut $M$ along embedded tori into simple and Seifert fibered pieces. A
graph manifold is a manifold that does not have any hyperbolic pieces in
its geometric decomposition, or equivalently only has Seifert fibered
pieces in its JSJ decomposition. To a graph manifold one can associate
several graph invariants, and in this paper we are going to show
properties of one of these called the splice diagram.

Splice diagrams were original introduced in \cite{EisenbudNeumann} and
\cite{Siebenmann} only for manifolds that are integer homology
spheres. Splice diagrams were then generalized to rational homology
spheres in \cite{NeumannWahl3} and used extensively in
\cite{neumannandwahl1}  and \cite{neumannandwahl2}. Our splice diagrams
differs from the ones in \cite{EisenbudNeumann} in that we do not allow
negative weights on edges, and from the ones in \cite{neumannandwahl1},
\cite{neumannandwahl2} and \cite{NeumannWahl3} in that we have
decorations on the nodes, but it is shown in
\cite{neumannandwahl2} that in the case of singularity links their splice
diagrams are the same as ours.   

Now it has long been known that the link of a isolated complex surface
singularity is a graph manifold who has a plumbing diagram with only
orientable base surfaces and negative definite
intersection form. Grauert showed that the inverse is also true. It is
shown in appendix 2 of \cite{neumannandwahl2} that a rational homology
sphere singularity link
has a splice diagram without any 
negative decorations at nodes and has all edge determinants positive.     
We here prove the other direction to get the following theorem.

\begin{thm*}\label{thm1}
Let $M$ be a rational homology sphere graph manifold with splice
diagram $\Gamma$. Then $M$ is a singularity link if and only if
$\Gamma$ has no negative decorations at nodes and all edge
determinants are positive.
\end{thm*}

Another interesting thing in the study of $3$ manifolds is knowledge
of the abelian covers of the manifold. Since our manifolds are rational
homology spheres their universal abelian covers are finite covers, and
we show that the splice diagram of a manifold determines its universal
abelian cover, more precisely we prove.

\begin{thm*}\label{thm2}
Let $M$ and $M'$ be rational homology sphere graph manifolds with the
same splice diagram $\Gamma$. Then $M$ and $M'$ have isomorphic
universal abelian covers.
\end{thm*}

In \ref{pre} we define our splice diagrams and prove
several facts we need about them, among other how one gets a
splice diagram from a plumbing diagram of the manifold. In \ref{det}
we show some important lemmas for our proofs, and that the splice
diagram together with the order of the first homology group determines
the decomposition graph of \cite{commensurability}. In \ref{prof} we
then prove the first theorem above. In \ref{graphorbifold} we define
graph orbifolds and show some of their properties, since we need them
in the proof of the second theorem which we prove in \ref{proof2}.   

The author wish to thank Walter Neumann for advice with the article.

\section{Preliminaries on Splice Diagrams}\label{pre}

A splice diagram is a tree which has vertices of valence one, which we
call leaves, and vertices of valence greater than or equal to $3$,
which are called nodes. The end of an edge adjacent to a node of the splice
diagram is decorated with a non negative integer, and each node is
decorated with either a plus or a minus sign; in general one only
writes the minus signs. Here is an example:   

$$\splicediag{8}{30}{
  \Circ&&&&\Circ&\\
  &\oplus  \lineto[ul]_(.25){3}
  \lineto[dl]^(.25){5}
  \lineto[rr]^(.25){22}^(.75){10}&& \ominus
  \lineto[ur]^(.25){7}
  \lineto[dr]_(.25){2}^(.75){6}&&\Circ\\
  \Circ&&&&\oplus\lineto[ur]^(.25){3}\lineto[dr]_(.25){2}&\\
&&&&&\Circ
\hbox to 0 pt{~,\hss} }$$

Given a rational homology sphere graph manifold $M$, one makes the
splice diagram $\Gamma(M)$ by taking a node for each Seifert fibered
piece. Then one attaches an edge between two nodes if they are glued
along a torus, and one adds leaves (vertex attached to a node along a
edge) to each node, one for each singular fiber of the Seifert fibered
piece corresponding to the node.

 The decoration $d_{ve}$ on edge $e$ at node $v$ one gets by
cutting the manifold along the torus corresponding to $e$, gluing
a solid torus in the piece not containing the Seifert fibered piece
corresponding to $v$, by
identifying a meridian of the solid torus with the fiber of in the Seifert
fibered piece and a longitude with a simple closed curve, which is
given by the JSJ decomposition. Then the $d_{ve}$
is the order of the first homology of this 
new manifold. We assign $0$ if the homology is infinite. 

For the decorations on nodes, we need the following definition.
\begin{defn}
Let $L_0,L_1\subset M$ be two knots in a rational homology sphere. Let
$C_1\subset M$ be a submanifold, such that $\partial C_1=d_1L_1$, for
some integer $d_1$. Then
the linking number of $L_0$ with $L_1$ is defined to be
$\lk(L_0,L_1)=\tfrac{1}{d_0}L_0\bullet C_1$. Where $\bullet$ denotes
the intersection product in $M$.
\end{defn} 
To see that $\lk(L_0,L_1)$ is well defined, we need to show that if
$C'_1\subset M$ is a submanifold such that $\partial C_1'=d'_1 L_1$
then $\tfrac{1}{d_0}L_0\bullet C_1=\tfrac{1}{d'_0}L_0\bullet
C'_1$, a $C_1$ always exist since $M$ is a rational homology sphere.
 Since $\partial C'_1=d'_1L_1$ we have that $\partial
(d_1C'_1)=d_1d'_1L_1$, in the same way we have that $\partial
(d'_1C_1)=d_1d'_1L_1$. We can thus form a closed submanifold
$N=d_1C'_1\bigcup_{d_1d_1'L_0}-d'_1C_1$. Since $M$ is a rational
homology sphere then the homology class of $N$ is $0$, so $L_0\bullet
N=0$. But then $0=L_0\bullet
N=L_0\bullet(d_1C'_1\bigcup_{d_1d_1'L_0}-d'_1C_1)=L_0\bullet d_1C'_1-L_0\bullet
d'_1C_1$. Since the intersection product is bilinear we get the
result by dividing by $d_1d'_1$.

To show that $\lk(L_0,L_1)=\lk(L_1,L_0)$, we will define another
notion of linking number to show that this is symmetric and equal to our
first definition.
\begin{defn}
Let $L_0,L_1\subset M$ be knots in a rational homology sphere, let $X$
be a compact $4$-manifold such that $M=\partial X$. Let
$A_0,A_1\subset X$ be submanifolds such that $\partial A_i=d_iL_i$ for
some integers $d_0,d_1$, and such that $A_1$ has zero intersection
with any $2$-cycle in $X$. Then let
$\widetilde{\lk}(L_0,L_1)=\tfrac{1}{d_0d_1}A_0\cdot A_1$, where $\cdot$
denotes the intersection product in $X$.
\end{defn}

$A_i$ exists since $M$ is a rational homology sphere, and we can the
choose $A_i\subset M$. That $A_0$ can be chosen so that it does not
intersect any closed cycles of $X$ follows because a collar
neighbourhood of $M$ has zero second homology, since $H_2((0,1]\times
M)\cong H_2(M)=\{0\}$, and hence we can just choose $A_0\subset
(0,1]\times M$. To show $\widetilde{\lk}(L_0,L_1)$ is well defined we
start by showing that if $A_1'\subset X$ is such that $\partial
A_1'=d_1'L_1$ then
$\tfrac{1}{d_0d_1}A_0\cdot A_1=\tfrac{1}{d_0d'_1}A_0\cdot A'_1$. We
form $N=(d_1'A_1\bigcup_{d_1d_1'L_1}-d_1A_1')$, then $A_0\cdot N=0$
since $N$ is a closed $2$-cycle, and it follows that
$\tfrac{1}{d_0d_1}A_0\cdot A_1=\tfrac{1}{d_0d'_1}A_0\cdot A'_1$ as
above. Now assume $A_0'\subset X$ is such that $\partial A_0'=d_0'L_0$
and $A_0'$ has zero intersection with all $2$-cycles in $X$. To show that
$\tfrac{1}{d_0d_1}A_0\cdot A_1=\tfrac{1}{d'_0d_1}A'_0\cdot A_1$ we can
choose a $A_1\subset X$ such that $A_1$ has zero intersection with any
$2$-cycles, since changing $A_1$ does not change
$\widetilde{\lk}(L_0,L_1)$ as just shown. Then form
$N'=(d_0'A_0\bigcup_{d_0d_0'L_0}-d_0A_0')$. Now $N$ is a $2$-cycle
and by our choice of $A_1$, $N\cdot A_1=0$ it follows that
$\tfrac{1}{d_0d_1}A_0\cdot A_1=\tfrac{1}{d'_0d_1}A'_0\cdot A_1$ and
therefor that $\widetilde{\lk}(L_0,L_1)$ is well defined. 

That $\widetilde{\lk}(L_0,L_1)$ is symmetric is clear since the
intersection product of $2$ dimensional submanifolds of a $4$
dimensional manifold is symmetric, and that it is the first of the two
that has zero intersection with the $2$-cycles does not matter since
we could have chosen both $A_0$ and $A_1$ to have zero intersection
with any $2$-cycles, making the definition symmetric. 
   
\begin{prop}
$\lk(L_0,L_1)=\widetilde{\lk}(L_0,L_1)$
\end{prop}
\begin{proof}
We choose $A_0$ such that in the collar neighborhood $(0,1]\times M$
we have that $A_0\cap(0,1]\times M=(0,1]\times d_0L_0$, and we choose
$A_1$ such that $A_1\subset M$. Then we get that
$\widetilde{\lk}(L_0,L_1)=\tfrac{1}{d_0d_1}(0,1]\times d_0L_0\cdot
A_1=\tfrac{1}{d_0d_1}\{1\}\times d_0L_0\cdot
A_1=\tfrac{d_0}{d_0d_1}L_0\bullet
A_1=\lk(L_0,L_1)$
\end{proof}
Hence $\lk(L_0,L_1)$ is symmetric.

To finish the decoration of the splice diagram we add at a node $v$ a
sign $\epsilon_v$ corresponding to the sign of the linking number of
two non singular fibers in the Seifert fibration. So by this we can
get that the splice diagram for the Poincare homology sphere is the
following
$$\splicediag{8}{30}{
  \circ&&\circ\\
  &\oplus\lineto[ul]_(.25){2}
  \lineto[ddd]^(.25){3}
   \lineto[ur]^(.25){5}&\\
&&\\
&&\\
  &\circ&\hbox to 0 pt{~,\hss} }$$

An edge between two nodes look like 
$$\splicediag{8}{30}{
  &&&&\\
  \Vdots&\overtag\Circ {v_0} {8pt}\lineto[ul]_(.5){n_{01}}
  \lineto[dl]^(.5){n_{0k_0}}
  \lineto[rr]^(.25){r_0}^(.75){r_1}&& \overtag\Circ{v_1}{8pt}
  \lineto[ur]^(.5){n_{11}}
  \lineto[dr]_(.5){n_{1k_1}}&\Vdots\\
  &&&&\hbox to 0 pt{~,\hss} }$$
and to such edges we assign a number called the edge determinant.

\begin{defn}
The edge determinant $D$ associated to a edge between nodes $v_0$ and
$v_1$ is defined by the following equation. 
\begin{align}\label{defedgedet}
 D=r_0r_1-\epsilon_0\epsilon_1 N_0 N_1.
\end{align}    
where $N_i$ is the product of all the edge weights adjacent to
$v_i$, except the one on the edge between $v_0$ and $v_1$, $r_i$ is
the edge weight adjacent to $v_i$ on the edge between $v_0$ and $v_1$,
and $\epsilon_i$ is the sign on the node $v_i$, if we interpret a
plus sign as $1$ and a minus sign as $-1$.
\end{defn}

One way of getting a splice diagram for a manifold is from a plumbing
diagram for the manifold. Suppose $M$ has a plumbing diagram $\Delta(M)$
which we assume is in normal form, which means that all base surfaces
are orientable and the decorations on strings are less that or equal
to $-2$. This is not quite the same normal form as in \cite{plumbing}
but using the plumbing calculus of that article we can get from
Neumann's normal form to the one we use see \cite{neumann88}. 

We then get a splice diagram $\Gamma(M)$ by taking one node for each
node in $\Delta(D)$, i.e. a vertex with more than 3 edges or genus
$\neq 0$. Since we are only working with rational homology spheres,
every vertex of the  plumbing diagram has genus $= 0$. Connect two
nodes in $\Gamma(D)$ if there is a string between the corresponding
nodes in $\Delta(M)$, and add a leaf at a node in $\Gamma(M)$ for
each string starting at that node in $\Delta(M)$ and not ending at any
node. 

If $\Delta(M)$ is a plumbing of a manifold we denote the intersection
matrix by $A(\Delta(M))$ and let $\det(\Delta(M))=\det(-A(\Delta(M)))$.  

\begin{lemma}
Let $v$ be a node in $\Gamma(M)$, and $e$ be a edge on that node. We
get the weight $d_{ve}$ on that edge by
$d_{ve}=\num{\det(\Delta(M_{ve}))}$, where $M_{ve}$ is the manifold
which has plumbing diagram corresponding to the piece not containing
$v$ if one cuts $\Delta(M)$ just after $v$ on the string corresponding
to $e$.
\end{lemma}

\begin{proof}
This follows since the absolute value of the determinant of the
intersection matrix of a rational homology sphere graph manifold is
the order of the first homology group. And that the manifold $M_{ve}$
is the manifold corresponding to the manifold one gets by gluing in a
solid torus to the boundary of the piece not containing $v$ after
cutting along the edge corresponding to $e$, by the gluing described
above. 
\end{proof}

\begin{lemma}\label{linkingsign}
Let $v$ be a node in $\Gamma(M)$. Then the sign $\epsilon$ at $v$ , is
$\epsilon=-\sign(a_{vv})$, where $a_{vv}$ is the entry of 
$A(\Delta(M))\inv$ corresponding to the node $v$. 
\end{lemma}

\begin{proof}
To prove this we calculate $\widetilde{\lk}(L_v,L_w)$ where $L_v$ is a
nonsingular fiber at the $v$'th node and $L_w$ is a non singular fiber at
the $w$'th node. Let $X$ be the plumbed $4$-manifold given by
$\Delta(M)$. Then each vertex of $\Delta(M)$ corresponds to a circle
bundle over a $2$-manifold in the plumbing so the $i$'th node gives us
a $2$-cycle $E_i$ in $X$, and the collection of all the $E_i$'s
generate $H_2(X)$. The intersection matrix $A(\Delta(M))$ is the
matrix representation for the intersection form on $H_2(X)$ in this
generating set. So to construct a $A_0$ such that it has zero
intersection with all $2$-cycles, we just need that $A_0\cdot E_i=0$
for all $i$. Let $D_v$ and $D_w$ be transverse disk to $E_v$ and $E_w$
with boundaries $L_v$ and $L_w$, if $v=w$ choose them disjoint. Set
$A_0=\det(\Delta(M))D_v-\sum_i\det(\Delta(M))(a_{vi})E_i$, where
$a_{ij}$ is the $ij$'th entry of $A(\Delta(M))\inv$, and choose
$A_1=D_w$. Then $A_0\cdot E_i=0$ for all $i$ and
$\widetilde{\lk}(L_v,L_w)=\tfrac{1}{\det(\Delta(M))}A_0\cdot
D_w=-\tfrac{1}{\det(\Delta(M))}\det(\Delta(M))(a_{vw})E_w\cdot
D_w=-a_{vw}$, since $E_i\cdot D_w=0$ if $i\neq w$ and $E_w\cdot D_w=1$. 
\end{proof}
The proof here is the same as given for proposition 9.1 in
\cite{neumannandwahl4}.

This way of obtaining a splice diagram shows that no edge weight on a
edge to a leaf is $0$, since we assumed that our plumbing diagram
is in normal form, especially that all weights on strings are $\leq -2$,
so a weight on an edge to a leaf is the determinant of a matrix on the form
\begin{align*}
\begin{pmatrix}
b_{11}&-1&0&\dots&0\\
-1&b_{22}&-1&\dots&0\\
0&-1&b_{33}&\dots&0\\
\vdots&\vdots&\vdots&\ddots&\vdots\\
0&0&0&\dots&b_{nn}\\
\end{pmatrix}
\end{align*}
where $b_{ii}\leq 2$, and determinants of such matrices are never $0$.

To prove the theorem we have to introduce another diagram, which we
will call the unnormalized splice diagram $\widetilde{\Gamma}(M)$. 

\begin{defn}
The unnormalized splice diagram $\widetilde{\Gamma}(M)$ is a tree,
with the same graph structure as the splice diagram $\gamma(N)$, but
it has no signs at nodes, and the weights at edges are defined to be
$\tilde{d}_{ve}=\det(\Delta(M_{ve}))$. 
\end{defn}

It is clear that one constructs the unnormalized splice diagram from
the plumbing diagram in the same way as the splice 
diagram, except that for the weight $\tilde{d}_{ev}$ on a edge $e$ at
the node $v$ one does not take the absolute value 
of the $\det(\Delta(M_{ve}))$, but just sets
$\tilde{d}_{ve}=\det(\Delta(M_{ve}))$, and one
does not put any signs at the nodes. 

\begin{lemma}
Let $\widetilde{\Gamma}(M)$ be an unnormalized splice diagram of the
rational homology sphere graph manifold $M$.
Then $\widetilde{\Gamma}(M)$ has the same form as $\Gamma(M)$ as a
graph, and for a node $v$ we get $d_{ve}=\num{\tilde{d}_{ve}}$ and
$\epsilon_v=\sign(\Delta(M))\prod_{e}\sign(\tilde{d}_{ve})$, where the
product is taken over all edges at $v$. 
\end{lemma}
\begin{proof}
That the graph has the same form and
$d_{ve}=\num{\tilde{d}_{ve}}$ is clear from the constructions.
The last follows from the proof of theorem 12.2 in
\cite{neumannandwahl2}. In their theorem they assume negative definite
intersection form, but if one looks at the proof one sees that for the
part we need it is not necessary to assume negative definiteness.
\end{proof}   

If we have a edge between two nodes in a unnormalized splice diagram
that look likes this 

$$\splicediag{8}{30}{
  &&&&\\
  \Vdots&\overtag\Circ {v_0} {8pt}\lineto[ul]_(.5){\tilde{n}_{01}}
  \lineto[dl]^(.5){\tilde{n}_{0k_0}}
  \lineto[rr]^(.25){\tilde{r}_0}^(.75){\tilde{r}_1}&& \overtag\Circ{v_1}{8pt}
  \lineto[ur]^(.5){\tilde{n}_{11}}
  \lineto[dr]_(.5){\tilde{n}_{1k_1}}&\Vdots\\
  &&&&\hbox to 0 pt{~,\hss} }$$
Then we the define the unnormalized edge determinant $\widetilde{D}$
associated to a edge to be
\begin{align}\label{defedgedet}
 \widetilde{D}=\tilde{r}_0\tilde{r}_1-\widetilde{N}_0 \widetilde{N}_1
\end{align}  
where $\widetilde{N}_i=\prod_{j+=}^{k_i}\tilde{n}_{ij}$.
Then it's clear that $D=\sign(\tilde{r}_0)\sign(\tilde{r}_1)\widetilde{D}$.  

We are also going to need what is called a maximal splice diagram, it
is a tree with integer weights on edges leaving vertices.

\begin{defn}
The maximal splice diagram of a manifold $M$ with plumbing diagram
$\Delta(M)$ has the underlying graph the graph of the plumbing
diagram. On edges one adds decorations  as in the construction of a
unnormalized splice diagram from the plumbing diagram $\delta(M)$.
\end{defn}

To get a unnormalized splice diagram from a maximal splice diagram, one
just removes the vertices of valence two and removes the decoration on
edges next to vertices of valence one.

An edge in our splice diagram between nodes $v_0$ and $v_1$
corresponds to a torus $T^2$ which the pieces corresponding to the
nodes are glued along. In that
torus we get several natural knots from the Seifert fibered structure
on each side, namely a fiber $F_i$ and a section $S_i$ from the
fibration of the piece corresponding to $M_i$. We are going to be
interested in the fiber intersection of $F_0$ with $F_1$ in the torus
$T^2$. We make the following convention, if we write $F_0\cdot F_1$ we
mean the intersection product in $T^2$, where $T^2$ is oriented as the
boundary of $M_0$ and when we write $F_1\cdot F_0$ we mean the
intersection product in $T^2$ oriented as the boundary of $M_1$. In
this way $F_0\cdot F_1=F_1\cdot F_0$, since we change the orientation
on $T^2$ when we interchange $F_0$ and $F_1$. 

Because $M$ is a rational homology sphere, the diagram is a tree.
It is then possible to orient the $F_i$'s and $S_i$'s such that the intersection
number of the fibers $F_i$ and $F_j$ from the Seifert fibered pieces
on each side of a separating torus is always positive, and so that
$F_i\cdot S_i=1$. It should be mentioned
that $S_i$ is only well defined up to adding a multiply of $F_i$, for
the case of orienting the $F_i$'s and the $S_i$'s, it does not mater. 

One does this be choosing a Seifert fibered piece
corresponding to a leaf and choose a orientation on the piece. This
then gives an orientation on the fiber in the boundary
of that piece and we the choose the right orientation of the
section. Choose a orientation on the piece glued along the torus, such
that the fiber intersection number is positive, and choose the right
orientation on the section. Then continue to do the same in the other
boundary pieces of this second Seifert fibered piece. We can then get
all the fibers and sections oriented this way, since $\Gamma(M)$ is a tree.  

We will always assume our fibers and sections are oriented this way. 

\section{Determine the decomposition graph from a splice diagram}\label{det}

Given a graph 3 manifold $M$ there is another graph invariant one can
associate to the JSJ decomposition of $M$ called the decomposition
graph. We will in this section show that given the splice diagram of a
manifold and the order of its first homology group, one can construct
the decomposition graph of that manifold.

The decomposition graph has one node for each Seifert fibered piece of
$M$, and a edge between nodes if they are glued by a torus. At node $v$
one puts 2 weights, the first is the rational euler number $e_v$ and
the other number is the orbifold euler characteristic of
the base $\chi^{orb}_v$. If the Seifert fibered piece corresponding to
the node $v$ has Seifert invariant
$M(g;(\alpha_1,\beta_1),\dots,(\alpha_k,\beta_k))$ then 
\begin{align}
e_v=-\sum_{i=1}^k\frac{\beta_i}{\alpha_i}
\end{align}
and
\begin{align}
\chi^{orb}_v=\chi_v-\sum_{i=1}^k(1-\frac{1}{\alpha_i})\label{orbifoldeuler}
\end{align}
where $\chi_v$ is the euler characteristic of the base surface. 
This formula is in fact only true for closed Seifert fibered spaces,
if the space has boundary, one needs additional information. The
additional information is a simple closed curve in each of the
boundary components, which we get
from the JSJ decomposition, by at each piece of the boundary take a
the curve corresponding to a fiber from the other side. Then one closes
the manifold by gluing in solid tori in the boundary pieces, by gluing
a meridian to the closed curve. Finally take
$e(v)$ to be the rational euler number of this closed manifold.

One weights an edge $e$ of the decomposition graph by the the intersection
number in $T$ of a nonsingular fibers of the Seifert fibrations on each side of
the torus $T$ corresponding to $e$. 

One gets the graph structure of the decomposition graph of $M$ from the splice
diagram $\Gamma(M)$ of $M$ by removing all leaves, i.e. by removing
all vertices of valence one and the edges leading to them. It is clear
that since nodes in the $\Gamma(M)$ corresponds to Seifert fibered
pieces in the JSJ-decomposition of $M$, and a edge between two nodes
means there are glued along a torus, that the result has the shape of the
decomposition graph. 

We start by given a formula for the orbifold euler characteristic
\begin{prop}\label{orbifoldeulerform}
Let $v$ be a node in the splice diagram $\Gamma(M)$ of the manifold
$M$. Then
\begin{align}
\chi^{orb}_v=2-n(v)+\sum_e\frac{1}{d_{ve}}
\end{align}
where $n(v)$ is the valence of $v$ and the sum is taken over all edges
leading to leaves.
\end{prop}

\begin{proof}
Since each leaf of the node $v$ corresponding to the Seifert fibered
piece corresponds to a singular fiber, and a singular fiber corresponds
to a leaf at $v$, so taking the sum in \eqref{orbifoldeuler} over
singular fibers is the same as taking the sum over edges at $v$
leading to leaves. The negative intersection matrix
$-A(\Delta(M_{ve}))$ has numbers $b_i\geq 2$ on the diagonal and $-1$
adjacent to diagonal entries and $0$ elsewhere since $e$ leads to a
leaf. We can then diagonalize $-A(\Delta(M_{ve}))$ only by adding rows
and columns in the following way. If the matrix is $n$ by $n$ we clear
the $1$ at the $(n,n-1)$ entry by adding $-\frac{1}{a_{nn}}$ times the
$n$'th row to the $n-1$'st row. Then we clear the $1$ at $(n-1,n)$ by
adding $-\frac{1}{a_{nn}}$ times the
$n$'th column to the $n-1$'st column. We now have that in the $n$'th
row and $n$'th column, only the diagonal entry is now zero. We then
proceed to clear the $(n-1,n-2$ and $(n-2,n-1)$ entries the same
way. This then continues until the matrix is diagonal.

If we do this we get that the $ii$'th entry of
$-A(\Delta(M_{ve}))$ is $[b_i,b_{i-1},\dots ,b_1]$ which is the
continued fraction 
\begin{align}
[b_i,b_{i-1},\dots ,b_1]=b_i-\cfrac{1}{b_{i-1}-\cfrac{1}{b_{i-2}-\dots}}.
\end{align}
Then $d_{ve}=\num{\det(\Delta(M_{ve}))}=\num{[b_n,b_{n-1}\dots
  ,b_1][b_{n-1},b_{n-2},\dots ,b_1]\cdots [b_1]}$. The continued
fraction $[b_1]=b_1$ and the denominator of $[b_i,,b_{i-1},\dots
,b_1]$ is the numerator of $[b_{i-1},,b_{i-2},\dots ,b_1]$. This
implies that $d_{ve}$ is equal to the numerator of $\num{[b_n,,b_{n-1},\dots
,b_1]}$. It follows from corollary 5.7 in \cite{plumbing}
that the numerator of $[b_n,b_{n-1},\dots ,b_1]$ is equal to $\alpha$,
where $\alpha$ is the first part of the Seifert
invariant of the singular fiber corresponding to the leaf at $e$. So
$d_{ve}=\alpha$ since $\alpha>0$.
We now have that
\begin{align}
\chi^{orb}_v=\chi_v-\sum_e(1-\frac{1}{d_{ve}})=\chi_v-l(v)+\sum_e\frac{1}{d_{ve}}.
\end{align}
where $l(v)$ is the number of singular fibers, which is the same as the
number of leaves at $v$. The base surface is a sphere since $M$ is a
rational homology sphere, so $\chi_v=2-r(v)$ where $r(v)$ is the
number of boundary components which is the same as the the number of
edges leading to other nodes. The formula the follows since $l(v)+r(v)=n(v)$. 
\end{proof}

We next proves a lemma relating the fiber intersection number to the
edge determinant.
\begin{lemma}[Unnormalized edge determinant
  equation]\label{nedgedeterminanteq} Assume that we have an edge in
  our splice diagram between two nodes. Let $T$ be the torus
  corresponding to the edge and $p$ the intersection number in $T$ of Seifert
  fibers from each of the sides of $T$. Let $d=\det(\Delta(M))$, then 
\begin{align}
p=\frac{\widetilde{D}}{d}
\end{align}
\end{lemma}

\begin{proof}
Let the numbers on the edge be as in

$$\splicediag{8}{30}{
  &&&&\\
  \Vdots&\overtag\Circ {v_0} {8pt}\lineto[ul]_(.5){n_{01}}
  \lineto[dl]^(.5){n_{0k_0}}
  \lineto[rr]^(.25){r_0}^(.75){r_1}&& \overtag\Circ{v_1}{8pt}
  \lineto[ur]^(.5){n_{11}}
  \lineto[dr]_(.5){n_{1k_1}}&\Vdots\\
  &&&&\hbox to 0 pt{~,\hss} }$$
And let $N_i=\prod_{j=1}^{k_i}n_{ij}$ for $i\in \{0,1\}$.

We start by proving the formula under the additional assumption that
there is no edge of weight $0$ adjacent to the nodes, except possibly
$r_0$ and $r_1$.

Let $H_i$ be a fiber at the $i$'th node. Let $L_i' \subset T^2$ be a simple
curve which generates $\ker(H_1(T^2,\Q)\hookrightarrow H_1(M_i,\Q))$ 
where $M_i$ is the piece
of $M$ gotten by cutting along the torus corresponding to the edge,
including the piece corresponding to the node $v_i$. 
Since $M$ is a rational homology sphere the Meyer Vietoris sequence
gives us that $H_1(T^2,\Q)\cong H_1(M_0,\Q)\bigoplus
H_1(M_1,\Q)$. $H_1(M_i,\Q)\cong H_1(T^2,\Q)/L_i$ by the long
exact sequence  
and $H_1(M_i/T^2)=H_1(M/M_{i+1})$ is a finite group. This implies
that $L_0$ and $L_1$ are linearly independent, so $L_0\cdot L_1\neq
0$, where $\cdot$ denotes the intersection product in $T^2$.  

We have the
following relation
\begin{align}
a_iH_i=b_{i0}L_0+b_{i1}L_1
\end{align} 
for some $a_i,b_{i0},b_{i1}\in \Z$, since $L_0,L_1$ are linearly
independent in $H^1(T^2,\Q)=\Q^2$ and hence a basis. We also note that
$L_i\cdot L_i=0$.  

We now want to compute the linking numbers $\lk(H_i,H_j)$ for $i,j\in
\{0,1\}$. Let
$C_i\subset M_i$ be such that $\partial C_i=
L_i$. This implies that $a_iH_i=b_{i0}\partial
C_0 + b_{i1}\partial C_1$. Then one can compute $\lk(H_i,H_j)$ as
 $\lk(H_i,\frac{1}{a_j}(b_{j0}\partial
C_0 + b_{j1}\partial C_1)$, but this is
the same as to compute $H_i\bullet(\frac{1}{a_j}(b_{j0} C_0 +
b_{j1}C_1)$, where $\bullet$ denotes the intersection number in
$M$. Now $C_0$ lives in the $M_0$ piece and $C_1$ in 
the $M_1$ piece, so when one computes
$H_0\bullet(\frac{1}{a_j}(b_{j0} C_0 + 
b_{j1}\frac{1}{c_1}C_1)$, it is only the
$C_0$ parts that matters, since $H_0$ is in the $M_0$ piece, and therefore
does not intersect things in the $M_1$ piece. This means we can
compute $\lk(H_0,H_j)$ as $H_0\bullet(\frac{1}{a_j}b_{j0} C_0)$.

$T^2$ has a collar neighborhood in $M_0$, so when we want to compute
$\lk(H_0,H_0)$ we can assume that the push-off of one of the copies of
$H_0$ in $T^2$ lives
in this collar neighborhood. I.e. if the collar neighborhood is
$(0,1]\times T^2$, then the push off is ${s}\times H_0$ for some $s\in
(0,1]$. Over the collar neighborhood $C_0$ is just $(0,1]\times
L_0$, so 
\begin{align*}
H_0\bullet(\frac{1}{a_0}b_{00}C_0)&=({s}\times H_0)\bullet
(\frac{1}{a_0}b_{00} ((0,1]\times  c_0L_0))\\
&= H_0\cdot (\frac{1}{a_0}b_{00}L_0)\\ 
&= \frac{1}{a_0}(b_{00}L_0+b_{01}L_1)\cdot (\frac{1}{a_0}b_{00}L_0)\\ 
&=\frac{1}{a_0^2}b_{01}b_{00}(L_1\cdot L_0).
\end{align*}
So we get that $\lk(H_0,H_0)=\frac{1}{a_0^2}b_{01}b_{00}(L_1\cdot
L_0)$. By a similar calculation one gets that
$\lk(H_0,H_1)=\frac{1}{a_0a_1}b_{01}b_{10}(L_1\cdot L_0)$ and
$\lk(H_1,H_1)=\frac{1}{a_1^2}b_{10}b_{00}(L_1\cdot L_0)$.

Another way to calculate the linking number of two fibers is that it's
given by the inverse intersection matrix. More precisely, the linking
number of a fiber at the $i$'th piece in a plumbing diagram
$\Delta(M)$ of $M$ with a fiber at the $j$'th piece is given by the
negative of the
$(i,j)$'th entry of $A(\Delta(M))\inv$, where $A(\Delta(M))$ is the
intersection matrix of the plumbing $\Delta(M)$, as we showed in the
proof of Lemma \ref{linkingsign}. By theorem 12.2 in
\cite{neumannandwahl2} it is equal to
$\frac{l_{ij}}{\det(\Delta(M))}$, where $l_{ij}$ is the product of
the weights adjacent to but not on the path from the $i$'th node to
the $j$'th node in the splice diagram and $n$ is the number of
vertices in $\Delta(M)$. In their theorem they are
calculating $l_{ij}$ in a maximal splice diagram, but
it is clear that if one has calculated the maximal splice diagram from
the plumbing diagram $\Delta(M)$, one gets our unnormalized
splice diagram from the maximal one, by removing any vertices with only 2
edges and not changing any weights. So if $i$ and $j$ represents
Seifert fibered pieces, then one gets the same number $l_{ij}$ by
calculating it in our unnormalized splice diagram, since no vertices
with only 2 edges can contribute to $l_{ij}$.

Returning to our situation, we then get the the following equations
for the linking numbers using the notation from
above. $\lk(H_0,H_0)=\frac{N_0r_0}{d}$,
$\lk(H_1,H_1)=\frac{N_1r_1}{d}$ and
$\lk(H_0,H_1)=\frac{N_0N_1}{d}$. Combining this with our other
equations for the linking numbers we get. 
\begin{align}
\frac{N_0r_0}{d}&=\frac{1}{a_0^2}b_{01}b_{00}(L_1\cdot
L_0)\label{h0h0}\\         
\frac{N_1r_1}{d}&=\frac{1}{a_1^2}b_{10}b_{11}(L_1\cdot
L_0)\label{h1h1}\\         
\frac{N_0N_1}{d}&=\frac{1}{a_0a_1}b_{01}b_{10}(L_1\cdot
L_0)\label{h0h1}        
\end{align}

It follows from \eqref{h0h1} that the $b_{ij}\neq 0$, since $N_i\neq 0$
by our assumptions. So we can divide the product of \eqref{h0h0} and
\eqref{h1h1} by \eqref{h0h1}, this gives us. 
\begin{align}
\frac{r_0r_1}{d}&=\frac{1}{a_0a_1}b_{00}b_{11}(L_1\cdot L_0)\label{r0r1}
\end{align}
Let us now compute $p$, which is equal to $H_0\cdot H_1$ by definition,
\begin{align*}
H_0\cdot H_1&=
\frac{1}{a_0}(b_{00}L_0+b_{01}L_1)\cdot\frac{1}{a_1}(b_{10}L_0+b_{11}L_1)\\  
&=\frac{1}{a_0a_1}(b_{01}b_{10}L_1L_0+b_{00}b_{11}L_0L_1)\\
&= \frac{r_0r_1-N_0  N_1}{d}\\
 &=\frac{\widetilde{D}}{d}   
\end{align*}
Here we use \eqref{h0h1}, \eqref{r0r1} and the definition of
$\widetilde{D}$.

We have now proved the the equality
\begin{align}
dp=r_0r_1-\prod_{i=1}^{k_0}n_{0i}\prod_{j=1}^{k_1}n_{1j}
\end{align}
whenever $n_{ij}\neq 0$. Now this is a equation concerning minors of
the negative intersection matrix $-A(\Delta(M))$ of $M$. We want to see
what happens if we vary the diagonal entries of $-A(\Delta(M))$. We are
especially interested in what happens when we change the entries of the
diagonal corresponding to changing one of the $n_{ij}$'s. Let $a$ be a
entry on the diagonal of $-A(\Delta(M))$ which lies in the minor
$n_{0l}$. Replacing $a$ with any integer $b$, we get a new
matrix $-A(\Delta(M_b))$, which is the intersection matrix of the graph
manifold $M_b$ one gets from a plumbing diagram corresponding to one for
$M$ with the weight corresponding to $a$ replaced with $b$. For all
values of of $b$, except maybe one, $M_b$ is a rational homology
sphere, since by computing $d=\det(\Delta(M_b))$ by expanding by the
row which include $b$, one gets $d=bA+B$ which has at most one
solution for $d=0$, because $aA+B\neq 0$. 

$M_b$ has splice diagram with the same form as the one for $M$,
in particular around the node we are working with it look like
$$\splicediag{8}{30}{
  &&&&\\
  \Vdots&\overtag\Circ {v_0} {8pt}\lineto[ul]_(.5){n_{01}^b}
  \lineto[dl]^(.5){n_{0k_0}^b}
  \lineto[rr]^(.25){r_0^b}^(.75){r_1^b}&& \overtag\Circ{v_1}{8pt}
  \lineto[ur]^(.5){n_{11}^b}
  \lineto[dr]_(.5){n_{1k_1}^b}&\Vdots\\
  &&&&\hbox to 0 pt{~,\hss} }$$

The only weights of the splice diagram of $M_b$ which are
different of the weights from the splice diagram of $M$, are $n_{0l}$
and $r_1$ since none of the others see the the entry of $-\Delta{M}_b$
which we have changed. Again $n_{0l}^b=bA_{01}+B_{01}$ and
$n_{01}=aA_{01}+B_{01}$, so $n_{0l}^b=0$ for at most one value of $b$. So
for all but maybe two values of $b$, we have the following equation
\begin{align}
d_bp=r_0r_1^b-n_{0l}^b\prod_{\substack{i=1 \\ i\neq
    l}}^{k_0}n_{0i}\prod_{j=1}^{k_1}n_{1j}. 
\end{align}
Let $\widetilde{N}_l=\prod_{\substack{i=1 \\ i\neq
    l}}^{k_0}n_{0i}\prod_{j=1}^{k_1}n_{1j}$. We get that $r_1^b=bA_1
+B_1$ and the above equation becomes
\begin{align}
(bA+B)p=r_0(bA_1+B_1)-(bA_{01}+B_{01})\widetilde{N}_l.
\end{align}
This is equivalent to 
\begin{align}\label{beq}
b(Ap-A_1r_0+A_{01}\widetilde{N}_j)=
Bp-B_1r_0+B_{01}\widetilde{N}_j. 
\end{align}
Since this is true for more than one value of $b$, it implies that
\begin{align}
Ap-A_1r_0+A_{01}\widetilde{N}_j=
Bp-B_1r_0+B_{01}\widetilde{N}_j=0. 
\end{align}
But this implies that equation \eqref{beq} holds for any value of
$b$. So the equation $dp=D$ holds even if we change the diagonal
entries of $-\Delta(M)$, so, in particular, it holds if some $n_{ij}=0$. Now
since we are only interested in rational homology spheres, and for
them $d\neq 0$, we get that the unnormalized edge determinant equation
always holds, by dividing by $d$.

\end{proof}

We get following corollary, just by using the relation between $D$ and
$\widetilde{D}$, that $\num{\det(\Delta(M))}=\num{H_1(M)}$ and taking
absolute value.
\begin{cor}[Edge determinant
  equation]\label{edgedeterminanteq}
For an edge between nodes in the splice diagram for the rational
homology sphere graph 
manifold $M$, we get the following equation
\begin{align}
p=\frac{\num{D}}{\num{H_1(M)}}
\end{align}
where $p$ is the intersection number in the torus corresponding to
the edge of a fiber from each of the sides of the torus, and $D$ is
the edge determinant associated to that edge.
\end{cor}

A consequence of the edge determinant equation is that no node in the
splice diagram can have more than one adjacent edge weight of value
$0$. This is because we know that no leaf has edge weight $0$, so if
we have a node with at least two adjacent edge weights of value $0$,
the edge determinant of an edge with $0$ on would be
$0r_1-\epsilon_0\epsilon_10N_1=0$, and then the edge determinant
equation implies that $p=0$. But $p=0$ means that the fibers from each
side of the torus corresponding to the edge have intersection number
$0$, but then we could extend the fibration over $T^2$. So the nodes
$v_0$ and $v_1$ 
correspond to one node $v$ corresponding to a Seifert fibered piece,
which would not be cut in the $JSJ$ decomposition.

Next we need a formula for computing the rational euler class
of the Seifert fibered pieces of our graph manifold, using only
information from the splice diagram and the order of the first
homology group. If we have a node in our splice diagram, as in Fig.\ 1
below, where everything to the left is leaves

$$\splicediag{8}{30}{
  &&&&\\
 &&& \overtag\Circ{v_1}{8pt}
      \lineto[ur]^(.5){m_{11}}
  \lineto[dr]_(.5){m_{1l_l}}&\Vdots\\
  \Circ&&&&\\ 
  \Vdots&\overtag\Circ {v} {8pt}\lineto[ul]_(.5){n_1}
  \lineto[dl]^(.5){n_k}
  \lineto[uurr]^(.25){r_1}^(.75){s_1}
  \lineto[ddrr]^(.25){r_k}^(.75){s_k}
  &\Vdots&&  \\
  \Circ&&&&\\
  &&& \overtag\Circ{v_k}{8pt}
      \lineto[ur]^(.5){m_{k1}}
  \lineto[dr]_(.5){m_{kl_k}}&\Vdots\\
  &&&&\\
  &&\text{Figure 1}&&\hbox to 0 pt{~,\hss} }$$
we let $N=\prod_{j=1}^kn_j$ and let $M_i=\prod_{j=1}^{l_i}m_{ij}$. Then
we have the following proposition.

\begin{prop}\label{eulernumber}
Let $v$ be a node in a splice diagram decorated as in Fig.\ 1 above with $r_i\neq
0$ for $i\neq 1$, let $e_v$ be the rational euler number of the 
Seifert fibered piece corresponding to $v$, then
\begin{align}  
e_v=-d\big(\frac{\epsilon s_1}{ND_1\prod_{j=2}^kr_k}+
\sum_{i=2}^k\frac{\epsilon_iM_i}{r_iD_i}\big) 
\end{align}
where $d=\num{H_1(M)}$ and $D_i$ is the edge determinant associated to
the edge between $v$ and $v_i$.
\end{prop}

\begin{proof}
We start by proving a formula for $e(v)$ using an unnormalized splice
diagram, and then show that the relation between unnormalized and
normalized splice diagram will give us the formula. 

We first assume that $r_1\neq 0$ and prove the formula under that
hypothesis.  

Let $\Gamma(M)$ be a non normalized splice diagram, looking like the
above one. It is constructed from the plumbing diagram $\Delta$, which
look like Fig.\ 2 below around the node $v$.

$$\splicediag{8}{30}{
  &&&&&&&\\
 \overtag\Circ{-b_{1a_1}}{8pt}\lineto[dr]
&&&&&& \overtag\Circ{c_1}{8pt}
      \lineto[ur]
  \lineto[dr] \lineto[dl]
  &\Vdots\\
  &\overtag\Circ{-b_{1(a_1-1)}}{8pt}\dashto[ddr]
  &&&&\overtag\Circ{-c_{1m_1}}{8pt}\dashto[ddl]&&\\    
  &&&&&&&\\
  &&\overtag\Circ{-b_{11}}{8pt}\lineto[dr]
  &&\overtag\Circ{-c_{11}}{8pt}\lineto[dl]
  &&&\\
  &&\Vdots&\overtag\Circ{b}{8pt}\lineto[dr]
  \lineto[dl]
  &\Vdots&&&  \\
  &&\overtag\Circ{-b_{l1}}{8pt}\dashto[ddl]
  &&\overtag\Circ{-c_{k1}}{8pt}\dashto[ddr]
  &&&\\
  &&&&&&&\\
  &\overtag\Circ{-b_{l(a_l-1)}}{8pt}\lineto[dl]
  &&&&\overtag\Circ{-c_{km_k}}{8pt}\lineto[dr]&&\\
  \overtag\Circ{-b_{la_l}}{8pt}
  &&&&&&\overtag\Circ{c_k}{8pt}
      \lineto[ur]
   \lineto[dr]
  &\Vdots\\
  &&&&&&&\\
  &&&\text{Figure 2}&&&&\hbox to 0 pt{~,\hss} }$$
where $b_{ij},c_{ij}\geq 2$.

We want to compute $\det(-\Delta)$, so we look at the negative intersection
matrix $-A(\Delta)$ of $\Delta$, which we can write like  

\begin{align*}
-A(\Delta)=
\begin{pmatrix}
b&-1&0&\dots&-1&\dots&-1&\dots&-1&\dots\\
-1&b_{11}&-1&\dots&0&\dots&0&\dots&0&\dots\\
0&-1&b_{12}&\dots&0&\dots&0&\dots&0&\dots\\
\vdots&\vdots&\vdots&\ddots&\vdots&&\vdots&&\vdots&\\
-1&0&0&\dots&b_{21}&\dots&0&\dots&0&\dots\\
\vdots&\vdots&\vdots&&\vdots&\ddots&\vdots&&\vdots&\\
-1&0&0&\dots&0&\dots&c_{11}&\dots&0&\dots\\
\vdots&\vdots&\vdots&&\vdots&&\vdots&\ddots&\vdots&\\
-1&0&0&\dots&0&\dots&0&\dots&c_{21}&\dots\\
\vdots&\vdots&\vdots&&\vdots&&\vdots&&\vdots&\ddots
\end{pmatrix}
\end{align*}

We get it the following way. If we delete the $b$-weighted vertex
$v\in\Delta$ we 
get $l$ components on the left and $k$ components on the right.

If we let $B_i$ be the negative intersection matrix of the $i$'th component to
the left. It is of the form  
\begin{align*}
B_i=
\begin{pmatrix}
b_{i1}&-1&\dots&0&0\\
-1&b_{i2}&\dots&0&0\\
\vdots&\vdots&\ddots&\vdots&\vdots\\
0&0&\dots&b_{ia_{i-1}}&-1\\
0&0&\dots&-1&b_{ia_i}
\end{pmatrix}.
\end{align*} 
Likewise we let $C_i$ be the negative intersection matrix of the
$i$'th component to the right.
\begin{align*}
C_i=
\begin{pmatrix}
c_{i1}&-1&\dots&0&0&\dots\\
-1&c_{i2}&\dots&0&0&\dots\\
\vdots&\vdots&\ddots&\vdots&\vdots&\vdots\\
0&0&\dots&c_{im_i}&-1&0\\
0&0&\dots&-1&c_i&-1\\
\vdots&\vdots&\dots&0&-1&\ddots
\end{pmatrix}
\end{align*} 
We get that $-A(\Delta)$ has $b$ at the $A_{11}$ entry, then the
$B_i$'s and the $C_i$'s are following along the diagonal. The first
row and column has a $-1$ in the column/row there corresponds to the
upper left corner of a $B_i$ or $C_i$, and all other entries $0$. 

\begin{align*}
-A(\Delta)=\left(
\begin{array}{c c c c c c c}
b &\begin{matrix}
 -1& 
\end{matrix}
&&\begin{matrix}
 -1& 
\end{matrix}&\begin{matrix}
 -1& 
\end{matrix}&&\begin{matrix}
 -1& 
\end{matrix}\\
\begin{matrix}
 -1\\ \quad 
\end{matrix}&
\text{\Huge $B_1$}
&&&&&\\
&&\ddots&&&&\\
\begin{matrix}
 -1\\ \quad 
\end{matrix}&&&
\text{\Huge $B_l$}
&&&
\\
\begin{matrix}
 -1\\ \quad 
\end{matrix}&&&&
\text{\Huge $C_1$}
&&\\
&&&&&\ddots&\\
\begin{matrix}
 -1\\ \quad 
\end{matrix}&&&&&&
\text{\Huge $C_k$}
\end{array}
\right)
\end{align*}

We will diagonalize the matrix to compute
$\det(-\Delta)=\det(-A(\Delta))$. This can be done by first
diagonalising the matrix, except for the first row and column. To see
how this is done we look at one of the $B_i$. We clear the off-diagonal
term in the last column by adding $\tfrac{1}{b_{ia_i}}$
times the last row from the second to last row, then we can clear the
$-1$ at the left of the $b_{ia_i}$ by a symmetric argument. We have
now cleared the off diagonal terms of the last row and column, and at
the second to last diagonal entry we have
$b_{ia_{i-1}}-\tfrac{1}{b_{ia_1}}$. Since all the $b_{ij}\geq 2$ we
can continue doing this using the last row and columns with
off-diagonal entries to clear the one before it. By diagonalising $B_i$ this
way we only add rows and columns, and we never use the first row and
column of $B_i$, this assures us that it does not change the matrix
outside the $B_i$ block, since the rows and columns we use have zeros
outside $B_i$. The first entry of the block  after diagonalising it
will then be
\begin{align*}
\beta_i=b_{i1}-\cfrac{1}{b_{i2}-\cfrac{1}{b_{i3}-\dots}}
\end{align*}  
We can also diagonalise the $C_i$'s in the same way, by starting at
the bottom right corner and working up, only adding rows
and columns that are not the first row and column. We will denote the
first entry of the diagonalization of $C_i$ by $\gamma_i$ 

To get the matrix completely diagonal we have to remove the
$-1$ in the first row and the first column. If $-1$ is in the first
row corresponds to the entry $\beta_i$ of a diagonalized $B_i$
 then we remove it by adding $\frac{1}{\beta_i}$
times the $i$'th 
row to the first. This changes the the first entry by subtracting
$\frac{1}{\beta_i}$. Similarly if the $-1$ corresponds to the entry
$\gamma_i$ of the diagonalized $C_i$ we let
\begin{align*}
\lambda_i=\gamma_i-c_{i1}-\cfrac{1}{c_{i2}-\cfrac{1}{c_{i3}-\dots}}
\end{align*}
Our assumptions on the splice diagram assures the $A_{ii}\neq 0$ since
$\det(C_i)=r_i$ and $A_{ii}\mid \det(C_i)$. We let
\begin{align*}
\xi_i=\cfrac{1}{c_{i1}-\cfrac{1}{c_{i2}-\cfrac{1}{c_{i3}-\dots}}}
-\cfrac{1}{c_{i1}-\cfrac{1}{c_{i2}-\cfrac{1}{c_{i3}-\dots}}+
  \lambda_i}
\end{align*}
So the change to the first entry of the matrix will be adding
\begin{align*}
\xi_i-\cfrac{1}{c_{i1}-\cfrac{1}{c_{i2}-\cfrac{1}{c_{i3}-\dots}}} 
\end{align*}

 We then get that first entry of the diagonalised matrix is
\begin{align*}
b-\sum_{i=1}^l\cfrac{1}{b_{i1}-\cfrac{1}{b_{i2}-\cfrac{1}{b_{i3}-\dots}}}
-\sum_{i=1}^k\cfrac{1}{c_{i1}-\cfrac{1}{c_{i2}-\cfrac{1}{c_{i3}-\dots}}} 
+\sum_{i=1}^k\xi_i
\end{align*}
Now we know that 
\begin{align*}
-e_v= b-\sum_{i=1}^l\cfrac{1}{b_{i1}-\cfrac{1}{b_{i2}-\cfrac{1}{b_{i3}-\dots}}}
-\sum_{i=1}^k\cfrac{1}{c_{i1}-\cfrac{1}{c_{i2}-\cfrac{1}{c_{i3}-\dots}}}
\end{align*}
by arguments of Walter Neumann in the proof of theorem 3.1 in
\cite{commensurability}. So if $d=\det(A(\Delta))$ we get that
\begin{align}
d=(-e_v+\sum_{i=1}^k\xi_i)
\prod_{i=1}^l\det(B_1)\prod_{i=1}^k\det(C_i)
\end{align}

But we also know that $n_i=\det(B_i)$ and $r_i=\det(C_i)$, so we get
the following formula. 
\begin{align}
d=(-e_v+\sum_{i=1}^k\xi_i)
\prod_{i=1}^ln_i\prod_{i=1}^kr_i=(-e(v)+\sum_{i=1}^k\xi_i)
N\prod_{i=1}^kr_i\label{d}
\end{align}

If we cut $M$ along the torus just before the Seifert fibered piece
corresponding to the node $c_j$, when coming from $v$, and glue in a
solid tori, we get a graph manifold with a non normalized splice
diagram $\Gamma'$ corresponding to $\Gamma$, where we remove
everything after $r_j$, so that $r_j$ becomes the weight corresponding to
a leaf. And the plumbing diagram $\Delta'$ of this manifold corresponds  
to $\Delta$ with everything after $c_{jm_j}$ removed. We can now make
the same calculation $\det(\Delta')$ as above and get that 
\begin{align}
\det(\Delta')=p_j(-e_v+\sum_{\substack{i=1 \\ i\neq j}}^k\xi_i)
N\prod_{\substack{i=1 \\ i\neq j}}^kr_i \label{sj}
\end{align}
where 
\begin{align*}
p_j=\det
\begin{pmatrix}
c_{i1}&-1&\dots&0\\
-1&c_{i2}&\dots&0\\
\vdots&\vdots&\ddots&\vdots\\
0&0&\dots&c_{im_i}\\
\end{pmatrix}
\end{align*}
But it follows from the proof of theorem 3.1 in
\cite{commensurability} that $p_j$ is the fiber intersection number for the
edge. By definition $\det(\Delta')=s_j$. So by combining \eqref{d} and
\eqref{sj} we get that  
\begin{align*}
-\xi_j=\frac{s_j}{p_jN\prod_{\substack{i=1 \\ i\neq
      j}}^kr_i}
-\frac{d}{N\prod_{i=1}^kr_i}=\frac{s_jr_j-dp_j}{p_jN\prod_{i=1}^kr_i} 
\end{align*}
by using that $p_j=\frac{\widetilde{D}}{d}$ by
\ref{nedgedeterminanteq} we get
\begin{align*}
-\xi_j=d\frac{s_jr_j-\widetilde{D}_j}{\widetilde{D_j}N\prod_{i=1}^kr_i} 
=d\frac{NM_j\prod_{\substack{i=1 \\ i\neq
      j}}^kr_i}{\widetilde{D_j}N\prod_{i=1}^kr_i}=\frac{dM_j}{r_j\widetilde{D}_j} 
\end{align*}
So using this in \eqref{d} we get
\begin{align}  
e(v)&=-d\big(\frac{1}{N\prod_{j=1}^kr_k}+
\sum_{i=1}^k\frac{M_i}{r_i\widetilde{D}_i}\big)\\
&=-d\big(\frac{\widetilde{D}_1}{N\widetilde{D}_1\prod_{j=1}^kr_k}+
\sum_{i=1}^k\frac{M_i}{r_i\widetilde{D}_i}\big)\\
&=-d\big(\frac{r_1s_1-N\prod_{j=1}^kr_kM_1}{N\widetilde{D}_1\prod_{j=1}^kr_k}+
\sum_{i=1}^k\frac{M_i}{r_i\widetilde{D}_i}\big)\\
&=-d\big(\frac{s_1}{N\widetilde{D}_1\prod_{j=2}^kr_k}+
\sum_{i=2}^k\frac{M_i}{r_i\widetilde{D}_i}\big)\label{neuler}
\end{align}

This proves the formula if $r_1\neq 0$. We get a new equation by
multiplying both sides of \eqref{neuler} by
$\prod_{i=1}^k\widetilde{D}_i\prod_{i=2}^kr_k$. This equation is as in the
proof of unnormalized edge determinant equation, an  equation in the
minors of $-\Delta(M)$. By changing a diagonal entry $b$ of $-\Delta(M)$,
lying in $C_1 $, we get that the equation becomes a polynomial
equation in $b$,
which is an equality for infinity many values of $b$, hence it is an
equality, so the equation holds for all value of $b$. We get our formula  by
dividing this equation by
$\prod_{i=1}^k\widetilde{D}_i\prod_{i=2}^kr_i$, which is not $0$  by
our assumption on the $r_i$'s.   

We saw earlier that $\epsilon=\sign(d)\prod_{i=1}^l\sign(n_i)
\prod_{i=1}^k\sign(r_i)$, we get that $\frac{d}{N\prod_{i=1}^kr_i}=
\frac{\epsilon\num{d}}{\num{N}\prod_{i=1}^k\num{r_i}}$. 

Using that
$D_i=\sign(r_i)\sign(s_i)\widetilde{D}$ and
$\epsilon_i=\sign(M_i)\sign(s_i)\sign(d)$, we also get that
$d\frac{M_i}{r_i\widetilde{D}_i}=
\num{d}\frac{\epsilon_j\num{M_i}}{\num{r_i}D_i}$, which proves the
proposition.

\end{proof}

From Corollary \ref{edgedeterminanteq} and Propositions
\ref{orbifoldeulerform} and
\ref{eulernumber} we get the information needed to make the
decomposition graph.

\section{Proof of the First Main Theorem}\label{prof}

In last section we saw that knowing the splice diagram and the order
of the first homology group lets us construct the decomposition graph
of $M$. It therefore also lets us construct the decomposition matrix,
also called the reduced
plumbing matrix as defined in \cite{commensurability}. By
theorem 3.1 in \cite{commensurability} we just need to show that the
decomposition matrix is negative definite if all edge determinants
are positive and the splice diagram has no negative decorations. 

\begin{thm}
Let $M$ be a rational homology sphere graph manifold, with splice
diagram $\Gamma$, such that all edge determinants are positive and
$\Gamma$ has no negative decorations at nodes. Then $M$ is a
singularity link.  
\end{thm}

\begin{proof}
The assumption that all edge determinants are $>0$ and we do not have
any negative decorations at edges assures that no edge weight is
$0$. Because if we had an edge weight of $0$, then it had to be on a
edge between nodes, but the edge determinant of this edge would be
$0r_1-\epsilon_0\epsilon_1N_0N_1=-N_0N_1<0$.

Let $d=\num{H_1(M)}$. 
We proceed by induction in the number of nodes of $\Gamma$. If
$\Gamma$ only has one node, then $M$ is Seifert fibered and the
reduced plumbing matrix is a $1\times 1$ matrix, with the rational
euler number $e$ of $M$ as it's entry. By Proposition \ref{eulernumber}, 
$e=-d\frac{\epsilon}{N\prod_{j=0}^kr_k}$. But $N,r_k,d$ are all
greater than $0$ by definition, and $\epsilon=1$ by assumption, so $e$
is negative. Hence the reduced plumbing matrix is negative definite.

Assume that there are $n$ nodes in $\Gamma$. Let $v$ be a end node of
$\Gamma$, meaning a node of the form

$$\splicediag{8}{30}{
  \Circ&&&&\\
  \Vdots&\overtag\Circ {v} {8pt}\lineto[ul]_(.5){n_1}
  \lineto[dl]^(.5){n_l}
  \lineto[rr]^(.25){r}^(.75){s}&& \overtag\Circ{v'}{8pt}
  \lineto[ur]^(.5){m_1}
  \lineto[dr]_(.5){m_k}&\Vdots\\
  \Circ&&&&\hbox to 0 pt{~,\hss} }$$
such nodes always exist since $\Gamma$ is a tree. Then the reduced
plumbing matrix is of the form 

\begin{align*}
\begin{pmatrix}
e(v)&\frac{1}{p}&0&\dots\\
\frac{1}{p}&e(v')&&\\
0&&\ddots&\\
\vdots&&&\\
\end{pmatrix}
\end{align*}

If we set $N=\prod_{i=0}^ln_i$ and $M=\prod_{i=0}^km_i$ we get
by Proposition \ref{eulernumber} that
\begin{align*}
e(v)=-d\frac{\epsilon s}{DN} 
\end{align*}
where $D$ is the edge determinant of the edge between $v$ and
$v'$. This means that the matrix look like this 
\begin{align*}
\begin{pmatrix}
\frac{-\epsilon sd}{DN}&\frac{1}{p}&0&\dots\\
\frac{1}{p}&e(v')&&\\
0&&\ddots&\\
\vdots&&&\\
\end{pmatrix}
\end{align*}
By a row and column operation we get the matrix to the form 
\begin{align*}
\begin{pmatrix}
\frac{-\epsilon sd}{DN}&0&0&\dots\\
0&e(v')+\frac{\epsilon DN}{p^2sd}&&\\
0&&\ddots&\\
\vdots&&&\\
\end{pmatrix}
=
\begin{pmatrix}
\frac{-\epsilon sd}{DN}
\end{pmatrix}
\oplus
\begin{pmatrix}
e(v')+\frac{\epsilon Nd}{Ds}&&\\
&\ddots&\\
&&&\\
\end{pmatrix},
\end{align*}
where the equality follows from using that $\frac{1}{p^2}=\frac{d^2}{D^2}$.
Since $s,d,N$ are positive by definition and $D,\epsilon$ are positive
by assumption, the reduced plumbing matrix is negative definite if the
matrix 
\begin{align*}
\begin{pmatrix}
e(v')+\frac{\epsilon Nd}{Ds}&&\\
&\ddots&\\
&&&\\
\end{pmatrix}
\end{align*}

is negative definite. Now
\begin{align*}
e(v')+\frac{\epsilon Nd}{Ds}=
-\frac{d}{Ms}-\sum_{i=0}^k\frac{\epsilon_iM_i}{r_iD_i}-\frac{\epsilon
  Nd}{Ds}+\frac{\epsilon
  Nd}{Ds}=-\frac{d}{Ms}-\sum_{i=0}^k\frac{\epsilon_iM_i}{r_iD_i}= \tilde{e}(v') 
\end{align*}
But $\tilde{e}(v')$ is the rational euler number of the Seifert
fibered piece corresponding to $v'$ in the manifold $M'$ which is the
manifold one gets by cutting $M$ along the edge between $v$ and $v'$,
and gluing in a solid tori in the piece containing $v'$. Then the matrix
\begin{align*}
\begin{pmatrix}
\tilde{e}(v')&&\\
&\ddots&\\
&&&\\
\end{pmatrix}
\end{align*}
is the reduced plumbing matrix for the manifold $M'$. But since the
splice diagram of $M'$ is the same as $\Gamma$ except it has a leaf
instead of the node $v$, it only has $n-1$ nodes. Then by induction the
reduced plumbing matrix of $M'$ is negative definite, so the reduced
plumbing matrix of $M$ is negative definite, and then by theorem 3.1
in \cite{commensurability} $M$ is the link of a complex surface
singularity. 

\end{proof}

\section{Graph orbifolds}\label{graphorbifold}

To prove the second main theorem we need to extend our notions and results
about graph manifolds. The reason is, that in the proof we are
doing induction on our graph manifold, which means we have to cut our
manifold along a torus and glue in solid tori to get some smaller
manifolds in which the statement holds by induction and whose
universal abelian cover contribute pieces to the universal abelian
cover of $M$. The problem is that
we do not always get manifolds when we glue in the solid torus. Already
in the simple case with the following plumbing diagram,
$$\splicediag{8}{30}{
  &&&\\
 \overtag\Circ{-2}{8pt}\lineto[dr]
&&& \overtag\Circ{-2}{8pt}
      \lineto[dl]\\
  &\overtag\Circ{-3}{8pt}\lineto[r]
  &\overtag\Circ{-3}{8pt}&\\    
  \overtag\Circ{-2}{8pt}\lineto[ur]
  &&&\overtag\Circ{-2}{8pt}\lineto[ul]\\
  &&&\hbox to 0 pt{~,\hss} }$$
the spaces one gets by gluing in the solid tori if one cuts along
the central edge are not manifolds. 

Fortunately the space we get when glue in the solid torus is not that
bad, and is what we will call a graph orbifold, which we define as
follows. 
\begin{defn}
 Let $M$ be a $3$ dimensional orbifold. We call $M$ a graph orbifold
 if there exist a collecting of disjoint smoothly embedded tori
 $T_i\subset M$, such that each connected component of $M-\bigcup T_i$
 is an $S^1$ orbifold bundle over orbifold surfaces.  
\end{defn}
It is clear that if a connected component of $M-\bigcup T_i$ is smooth,
then it is a Seifert fibered manifold, hence if $M$ is smooth it is a
graph manifold. 

We want to define the splice diagram of a rational homology graph
orbifold. But to do this we have to consider which homology we are
going to use. Remember that if $M$ is smooth then
$\pi_1^{orb}(M)=\pi_1(M)$ where $\pi_1^{orb}(M)$ is the orbifold
fundamental group defined by Thurston see e.g.\ \cite{scott}. So in
the case of smooth 
manifolds orbifold coverings and coverings are the same. We need
orbifold coverings, and the interesting homology group is then
$H_1^{orb}(M)$, which 
for our purpose it is enough to define as the abelianization of
$\pi_1^{orb}(M)$, since it governs the abelian orbifold covers of
$M$. It should be mentioned that there exists a de Rham theorem for
orbifold cohomology with rational coefficients, which says that
$H_{orb}^*(X;\Q)\cong H^*(X;\Q)$. So an orbifold is a rational
homology sphere as an orbifold if and only if its underlying space is 
a rational homology sphere. Orbifolds
also satisfies Poincare duality with rational coefficients. See e.g.\
\cite{orbifold} for these results. 

Next we look at the decomposition of a graph orbifold $M$ into fibered
pieces. To have 
a unique decomposition we do it the following way. Start by removing
a solid torus neighborhood of each orbifold curve 
$K_j\subset N_j\subset M$, i.e.\ a curve along which $M$ is not a manifold. Let
$M'=M-\bigcup_{j=1}^m N_j$, then $M'$ is a graph manifold with $m$
torus boundary components. We then take the JSJ decomposition of $M'$,
and  glue 
the $N_j$'s back in the pieces of the JSJ decomposition of $M$. This
give us our decomposition of $M$ into fibered pieces. It is unique
since the JSJ decomposition of $M'$ is unique. 



To define the splice diagram $\Gamma(M)$ of a graph orbifold $M$,
we start by taking a node for each connected component of
$M-\bigcup_{i=1}^n T_i$, where the set of $T_i$ comes from the
decomposition we defined above.
We then connect two nodes in
$\Gamma(M)$ if the corresponding connected components of
$M-\bigcup_{i=1}^n T_i$ are glued along a torus. We add a leaf at a
node for each singular fiber of the $S^1$ orbifold bundle over the
orbifold surface $\Sigma$, this is the same as adding a leaf for each
point in $\Sigma$ which does not have trivial isotropy group.

To put decorations on $\Gamma(M)$ we do the same as in the splice
diagram, except that where for a manifold we used the first singular
homology group, we now use the first orbifold homology group. That
is, to
get decorations on a edge we cut $M$ along the corresponding torus,
glue in a solid torus in the same way as for manifolds and take the
order of the first orbifold homology group of the new graph orbifold
as the decoration, and at nodes we put the sign of the linking number
of two non singular fibers of the $S^1$ fibration corresponding to the
node.

Let us take a closer look at the orbifold curves. In any 3 dimensional
orbifold $M$, an orbifold curve $K\subset M$ is a embedding of $S^1$
such that there exist a neighborhood $N_K$ of $K$ and $N_K-K$ is
smooth. Now $N_K$ can be chosen to be topological a solid torus, and in this
case $H_1^{orb}(N_K)=\Z\oplus \Z/p\Z$. We call $p$ the orbifold degree
of $K$. Another way to view $N_K$ is as a $S^1$ fibration over a disk $D_\alpha$
with one orbifold point where the isotropy group is
$\Z/\alpha\Z$. Then $N_K$ is defined by a integer $\beta$ which tells
you how much the fibers over the non orbifolds points wrap around the
singular fiber over the orbifold point. Now if $\gcd(\alpha,\beta)=1$
then $N_K$ is in fact smooth, and $K$ is not a orbifold curve,
but a singular fiber of the Seifert fibration in a neighborhood of
$K$. In general a calculation shows that if $K$ is a orbifold curve of
degree $p$, then $\gcd(\alpha,\beta)=p$.

 Next look at how $N_K$ is glued to
$M'=\overline{M-N_K}$. We have a collar neighborhood
$U=(0,1]\times T^2$ of $\partial (M')$. The fibration of $N_K\to
D_\alpha$ gives a fibration $\partial N_K\to S^1$, this again gives a
fibration of $\partial U$ which can be extended to all of $U$. The
image of a meridian $N_K$ in $\partial U$ defines a simple closed
curve transverse to the fibration. By a meridian of $N_K$ we mean a
simple closed curve of the boundary, that is transverse to the
fibration and has homology class of finite order in $H_1^{orb}(N_K)$. The
fibration on $U$ and the simple closed curve transverse to the
boundary uniquely describe a way to glue in a solid torus in the
boundary of $U$ to make it an Seifert fibered manifold, we call this
manifold $M_K$. Note that the gluing maps $\morf{\phi}{\partial
  M'}{\partial N_K}$ and $\morf{\phi'}{\partial
  M'}{\partial (S^1\times D^2)}$ are the same, and it is therefore also
clear that as topological spaces $M$ and $M_K$ are the same.   
   
\begin{prop}\label{orbhom}
Let $K\subset M$ be an orbifold curve of degree $p$ in a rational
homology sphere orbifold $M$. Then $\num{H_1^{orb}(M)}=p\num{H_1^{orb}(M_K)}$. 
\end{prop}

\begin{proof}
We get the following exact sequence from the Meyer-Vietoris sequence
of the cover of $M$ by $M'$ and $N_K$ 
\begin{align}
0\to\Z^2\xrightarrow{i_*} H_1^{orb}(M')\oplus \Z\oplus\Z/p\Z\to H_1^{orb}(M)\to 0
\end{align}
by using that $H_1^{orb}(N_K)=\Z\oplus \Z/p\Z$. The zero follows since
$M$ is a rational homology sphere, hence $H_2^{orb}(M)$ has to be
finite, and therefore have image zero in $\Z^2$. We likewise get the
exact sequence
\begin{align}
0\to\Z^2\xrightarrow{i_*'} H_1^{orb}(M')\oplus \Z\to H_1^{orb}(M_K)\to 0
\end{align}
from the Meyer-Vietoris sequence of $M_K$ by the cover of $M'$ and
$S^1\times D^2$. Now $i_*=(i_*',g)$ where the image of $g$ is in
$\{0\}\times\Z/p\Z$ from the way we constructed $M_K$
above. We have the following maps $\morf{\pi}{H_1^{orb}(M')\oplus
  \Z\oplus\Z/p\Z}{H_1^{orb}(M')\oplus \Z}$ given by
$\pi(a,b,c)=(a,b)$ and $\morf{f_*}{H_1^{orb}(M)}{H_1^{orb}(M_K)}$ which
  is the map induced on orbifold homology groups, by the homeomorphism
  $\morf{f}{M}{M_K}$ which is the identity on the compliment of
  $N_K$. Note that $\morf{f\vert_{M'}}{M'}{M'}$ is the identity and
  $\morf{f\vert_{N_K}}{N_K}{S^1\times D^2}$ induces the map from
  $\Z\bigoplus\Z/p\Z$ to $\Z$ given by $(b,c)=b$, so we have the
  following map of short exact sequences
\begin{align}\label{sec1}
\xymatrix{
0\ar[r]&\Z^2\ar[r]^-{i_*}\ar[d]^\cong &H_1^{orb}(M')\oplus
\Z\oplus\Z/p\Z\ar[r]\ar[d]^\pi &H_1^{orb}(M)\ar[r]\ar[d]^{f_*} &0\\ 
0\ar[r] &\Z^2\ar[r]^-{i_*'} &H_1^{orb}(M')\oplus \Z\ar[r] &H_1^{orb}(M_K)\ar[r] &0
}.
\end{align}
Using the snake lemma on \eqref{sec1} we get the following short exact
sequence.
\begin{align} 
0\to\Z/p\Z\to H_1^{orb}(M)\to H_1^{orb}(M_K)\to 0
\end{align}
and since this is a short exact sequence of finite abelian groups, the
order of the group in the middle is the product of the order of the
other two groups, which proves the theorem.
\end{proof}

\begin{cor}
Let $M$ be a rational homology sphere graph orbifold and
$\overline{M}$ be its underlying topological manifold then
$\num{H_1^{orb}(M)}=P\num{H_1(\overline{M})}$, where $P$ is the
product of the degrees of all orbifold curves in $M$. 
\end{cor}
 
We say that a edge weight $r$ of a splice diagram sees a leaf $v$ of
that splice diagram if, when we delete the node which $r$ is adjacent
to, the leaf $v$ and the edge which $r$ is on are in the same connected
component. 

\begin{cor}
The splice diagram $\Gamma(M)$ is equal to the splice diagram
$\Gamma(\overline{M})$ except if a edge weight sees a leaf
corresponding to an orbifold curve of $M$ it is multiplied by the
degree of the orbifold curve. 
\end{cor}

\begin{proof}
If an edge weight $r$ sees a leaf then the orbifold curve
corresponding  to that leaf is in the orbifold piece whose order
of the first homology group gives $r$.
\end{proof} 

\begin{cor}\label{cedgedeterminanteq}
Assume that we have an edge in
  our $\Gamma(M)$ between two nodes. Let $T$ be the torus
  corresponding to the edge and $p$ the intersection number in $T$ of
  non-singular
  fibers from each of the sides of $T$. Let $d=\num{H_1^{orb}(M)}$, then 
\begin{align}
p=\frac{\num{D}}{d}
\end{align}
where $D$ is the edge determinant of that edge.
\end{cor}

\begin{proof}
Since $T$ is in the smooth part of $M$, the fiber intersection number
is the same in $M$ and $M'$, and hence the same in $\overline{M}$. The
equation holds in $\overline{M}$ by \ref{edgedeterminanteq}, and, since
each term of $D$ sees each orbifold curve once, it holds in  $M$.
\end{proof}

\begin{cor}\label{ceulernumber}
Let $v$ be a node in a splice diagram decorated as in Fig.\ 1
with $r_i\neq 
0$ for $i\neq 1$, let $e(v)$ be the rational euler number of the 
$S^1$ fibered orbifold piece corresponding to $v$, then
\begin{align}  
e(v)=-d\big(\frac{\epsilon s_1}{ND_1\prod_{j=2}^kr_k}+
\sum_{i=2}^k\frac{\epsilon_iM_i}{r_iD_i}\big) 
\end{align}
where $d=\num{H_1^{orb}(M)}$ and $D_i$ is the edge determinant associated to
the edge between $v$ and $v_i$.
\end{cor}

\begin{proof}
Since the rational euler number associated to the node is the same in
$M$ and $\overline{M}$ and the formula holds for $\overline{M}$ by
\ref{eulernumber}, it
follows by noticing that the same orbifold degrees show up in the
numerator and the denominator. 
\end{proof}

\section{Proof of the Second Main Theorem}\label{proof2}

A splice diagram for a Seifert fibered manifold $M$, which has Seifert
invariant $(0;(\alpha_1,\beta_1),\dots ,(\alpha_n,\beta_n))$ look
like this
$$\splicediag{8}{30}{
   \Circ&&\\
  \Vdots&\overtag\Circ {v} {8pt}\lineto[ul]_(.5){\alpha_1}
  \lineto[dd]_(.5){\alpha_{n-2}}
  \lineto[r]^(.5){\alpha_n}
   \lineto[dr]_(.5){\alpha_{n-1}}&\Circ\\
 &&\Circ\\
 &\Circ& 
\hbox to 0 pt{~,\hss} }$$
and by \ref{eulernumber} the sign of the rational euler number is
equal to minus the sign at the node of the splice diagram. So from the splice diagram, we can read off the
$\alpha_i$'s and the sign of the rational euler number, but this is
exactly the information that determines the universal abelian cover
of $M$. By theorem 8.2 in \cite{neumann832} the universal abelian cover
of $M$ is homeomorphic to the Brieskorn complete intersection
$\sum(\alpha_1,\dots,\alpha_n)$ provided $e<0$. If $e>0$ one has to
compose with a orientation reversing map. The case $e=0$ does not occur for a
rational homology Seifert fibered manifold. We will generalize this
this result to graph manifolds. But to do this we need to prove it for
graph orbifolds. 

A $S^1$ fibered orbifold will have a splice diagram as above in the
case of Seifert fibered manifolds, and its universal abelian cover
will likewise be $\sum(\alpha_1,\dots,\alpha_n)$, but now it follows
from the proof of the above theorem in \cite{geometryss} which also
works when $\gcd(\alpha_i,\beta_i)\neq 1$.

Now this will prove the induction start in most cases but to prove it in
general we need the following lemma.

\begin{lemma}\label{connectedsum}
  Let $\morf{\pi_1}{M}{M_1}$ and $\morf{\pi_2}{M}{M_2}$ be universal
  abelian orbifold covers such that $\deg(\pi_1)=\deg(\pi_2)=d$ and
  both $M_1$ and $M_2$ have an orbifold curve of degree $p$. Let
  $L(n,m_1)$ and $L(n,m_2)$ be orbifold quotients of $S^3$ by $\Z/n\Z$
   which contain orbifold curves of degree
  $p$. Then the universal abelian cover of $L(n,m_1)\#_pM_1$ is
  homeomorphic to $L(n,m_2)\#_pM_2$, where $\#_p$ means taking
  connected sum along $B^3$ that intersects the orbifold curve of
  degree $p$, and the degree of the cover is $nd/p$.
\end{lemma} 
\begin{proof}
We are going to prove the lemma by constructing the universal abelian
cover of $L(n,m_i)\#_pM_i$, and see it is determined by $M,n,d$ and
$p$. 

Let $B^2_p\subset M_i$ be the ball with a orbifold curve of degree $p$
passing through which we are going to remove to take connected sum. Let
$M'_i=M_i-B^3_p$ and $S^2_{p}=\partial M'_i$. $\pi\inv
(M'_i)=\widetilde{M}_i$ is connected submanifold of $M$ and
$\morf{\pi_1\vert_{\widetilde{M}_i}}{\widetilde{M}_i}{M_i}$ is an abelian
  cover. Now $\pi_i$ restricted to a connected component of
  $\partial\widetilde{M}_i$ is
  the $p$-fold cyclic branched cover of $S^2$, hence the number of
  boundary components of $\widetilde{M}_i$ is $d/p$. So clearly
  $\widetilde{M}_i$ is homeomorphic to $M$ with $d/p$ balls removed,
  and hence does not depend on $M_i$ and $\pi_i$.  If we then look at
  $S^2_{p}=\partial (L(n,m_i)-B^3_p)$ then the preimage under the
  universal abelian cover of $\morf{p_i}{S^3}{L(n,m_i)}$ of
  $L(n,m_i)-B^3_p$, is $S^3 $  with $n/p$ balls removed. Let
  $\widetilde{M}$ be the manifold constructed the following way. Take
  $n/p$ copies of  $\widetilde{M}_i$ and $d/p$ copies of $S^3$ with $n/p$ balls
  removed. Then glue each of the $\widetilde{M}_i$ to each of the
  $S^3$'s exactly once, to form $\widetilde{M}$. Since $\pi_i$ and
  the universal abelian cover map from $S^3$ to $L(n,m_i)$ agrees on
  boundary components, we get an abelian cover
  $\morf{\tilde{\pi}_i}{\widetilde{M}}{L(n,m_i)\#_pM_i}$ of degree
  $nd/p$, by letting $\tilde{\pi}_i$ be equal to $\pi_i$ on each of
  the $\widetilde{M}_i$ components and to the $p_i$ on the $S^3$
  components. 

Using Meyer-Vietoris sequence we get that
$\num{H_1^{orb}(L(n,m_i)\#_pM_i)}=nd/p$ hence
$\morf{\tilde{\pi}_i}{\widetilde{M}}{L(n,m_i)\#_pM_i}$ is the
universal abelian cover, which proves the lemma since the
homeomorphism type for $\widetilde{M}$ only depends on $M,n,d$ and $p$. 
\end{proof}

Remark that the above theorem is not true if we took connected sum
along spheres with different degrees of the orbifold points, e.g\
$L(6,3)\#_3L(6,3)$ has universal abelian cover $S^1\times S^2$, but
the universal abelian cover of $L(6,3)\#L(6,3)$ has first homology
group of rank $25$. Since the splice diagram cannot see the orbifold
curve we are going to take connected sum along, this forces us to make
the assumption on our graph orbifolds in the theorem
below. Alternatively one could make a new definition of splice diagram,
where at leaves of weight zero one  specifies the
degree of the orbifold
curve in the solid torus corresponding to the leaf. The theorem then
holds for all graph orbifolds with this invariant. 

\begin{thm}\label{universalabcover}
If $M$ and $M'$ are two rational homology sphere graph orbifolds having
the same splice diagram $\Gamma$, and assume that all solid tori
corresponding to leaves of weight zero do not have orbifold curves.
Then $\widetilde{M}$ is homeomorphic
to $\widetilde{M}'$, where $\morf{\pi}{\widetilde{M}}{M}$ and
$\morf{\pi'}{\widetilde{M}'}{M'}$ are the universal abelian orbifold
covers.  
\end{thm} 

\begin{proof}
We will prove the theorem by inductively constructing $\widetilde{M}$
only using information from the splice diagram. 

For the case with one node, this is mostly the theorem from
\cite{neumann832} and \cite{geometryss} cited above, since every
one-node graph orbifold is a 
$S^1$ fibered orbifold, if we have no edge weight of $0$. So we have
to consider the case of a one-node splice diagram with a edge weight
of $0$.

Let $M$ be a orbifold with the following splice diagram

$$\splicediag{8}{30}{
  \Circ&&\\
  \Vdots&\overtag\Circ {v_0} {8pt}\lineto[ul]_(.5){n_1}
  \lineto[dl]^(.5){n_k}
  \lineto[r]^(.25){0}& \Circ\\
  \Circ&&\hbox to 0 pt{~.\hss} }$$
%
%
This orbifold is a $S^3$ connect summed along smooth $S^2$'s with the
orbifold quotients of $S^3$ by $\Z/n_i\Z$
$L(n_1,q_1),\dots, L(n_k,q_k)$, where the pair $(n_i,q_i)$ is the
Seifert invariant of the $i$'th singular fiber. 

One sees this the following way. The leaf with edge weight
zero means that the fibers of the piece corresponding to the central
node bounds a meridional disc in the solid torus $Z$ corresponding to the
leaf of weight zero. Take two fibers $F_1$ and $F_2$ and a simple path $p$
between them in the boundary of the $Z$. Then the region $B\subset Z$
bounded by
all the fibers intersecting $p$ and the meridional discs bounded by
$F_1$ and $F_2$ is a ball. We can now extend $B$ to the boundary of
the solid torus $L$ corresponding to the leaf of weight $n_i$. So by this,
part of the boundary of $B$ is a annulus of fibers in $\partial
L$. Now $L\bigcup B$ is a solid torus glued to a ball
along a strip cross a knot which is a representative of a non trivial
homology class of the boundary of $L$, and hence clearly has
boundary $S^2$. Another way to see this is that the boundary is an
annulus union 2 discs. $L\bigcup B$ includes a singular fiber, hence it is
not a ball, and therefore the $S^2$ is a separating sphere, and
$M=(L\bigcup B)\# (M-L\bigcup B)$. 

What is left is just to see what
$L\bigcup B$ is. To see this we see that the complement of $B$ in $Z$
is a ball. So gluing this ball to $L\bigcup B$ we get the
same orbifold as if we glued $L$ to $Z$, and since $Z$ does not have any
singular fibers it is a quotient of $S^3$ with a orbifold curve with
Seifert invariant $(n_i,q_i)$.      

So by doing this for each of the leaves with non zero weight, we get
that $M$ is connected sum of the $S^3$ quotients $L(n_1,q_1),\dots,
L(n_k,q_k)$ and a central piece $M'$. What is left is to see that the
central piece is $S^3$. If we glue a ball in $M-L\bigcup B$ to make the
closed manifold $M'$, we see that $M'$ is $Z$ glued to a solid torus,
and the gluing map is the same as when we glued $Z$ to $M-Z$. Since
the weight of the leaf corresponding to $Z$ was zero, it implies that
a fiber of $T^2=\partial(M-Z)$ is a generator of $H_1(T^2)$ and is glued
to a meridian of $Z$ and a simple closed curve $c$ corresponding to
the other generator is glued to a longitude of $Z$. This is because
that is how one specifies the gluing of $Z$ to a solid torus to get
 the decoration of the splice diagram, the glued manifold here being
 $S^1\times S^2$ since the weight is zero. But
gluing two solid tori according to the gluing of
$M$ and $Z$ described above creates a $S^3$. 

To show that the universal abelian cover of $M$ is determined by the splice
diagram, we do induction in the number of $S^3$ quotients in $M$, i.e.\ the
number of leaves of the splice diagram of $M$. If there is only one
$S^3$ quotient, then $S^3$ connect sum $L(n,q)$ is just $L(n,q)$, so the
universal abelian cover of $M$ is just $S^3$ and the covering map has
degree $n$, hence determined by the splice diagram.

Let $M'$ be the connected sum of $S^3$ with $L(n_1,q_1),\dots,
L(n_{k-1},q_{k-1})$. Then $M'$ has splice diagram as follows
$$\splicediag{8}{30}{
  \Circ&&\\
  \Vdots&\overtag\Circ {v_0} {8pt}\lineto[ul]_(.5){n_1}
  \lineto[dl]^(.5){n_{k-1}}
  \lineto[r]^(.25){0}& \Circ\\
 \Circ&&\hbox to 0 pt{~.\hss} }$$
So by induction the universal abelian cover $\widetilde{M}'$ of $M'$
and the degree for this universal abelian cover $d$ is determined by
the splice diagram. We also have that $M=L(n_k,q_k)\#M'$ so by Lemma
\ref{connectedsum} the universal abelian cover of $M$ is determined by
$\widetilde{M}'$, $n_k$ and the degree of the cover $\widetilde{M}'\to
M'$ (remember in this case $p=1$). But all this information is given
by the splice diagram, since the splice diagram of $M'$ is determined
by the splice diagram of $M$.

This completes the one node case. For more than one node we will
reduce our case to one with fewer nodes by cutting along a torus in
$M$ corresponding to an edge joining two nodes in $\Gamma$. This is
more complicated than when we cut along a sphere
since what we cut along is now not
simply covered once by itself, but may be multiply covered, and the
gluing of $2$--torus boundary components is not trivial as it is with
$2$--spheres.

Let us assume that $M$ has splice diagram $\Gamma$ with $n>1$ nodes. We
look at a edge between two nodes of the form

$$\splicediag{8}{30}{
  &&&&\\
  \Vdots&\overtag\Circ {v_0} {8pt}\lineto[ul]_(.5){n_{01}}
  \lineto[dl]^(.5){n_{0k_0}}
  \lineto[rr]^(.25){r_0}^(.75){r_1}&& \overtag\Circ{v_1}{8pt}
  \lineto[ur]^(.5){n_{11}}
  \lineto[dr]_(.5){n_{1k_1}}&\Vdots\\
  &&&&\hbox to 0 pt{~.\hss} }$$

Let $T^2\subset M$ be the separating torus, corresponding to the edge
we have chosen. Let $M_i^\circ\subset M-T^2$ be the connected component
containing the node $v_i$, and let $M_i=M_i^\circ\cup T^2$. $M_0$ and $M_1$
are graph orbifolds with one boundary torus
each. $\morf{\pi\lvert_{\pi\inv(M_i)}}{\pi\inv(M_i)}{M_i}$ is a
  (possibly disconnected) abelian covering. Let
  $\widetilde{M}_i\subset \pi\inv(M_i)$ be a connected component. Then
  $\morf{\pi\lvert_{\widetilde{M}_i}}{\widetilde{M}_i}{M_i}$ is a
  connected abelian covering.

To describe the covering
$\morf{\pi\lvert_{\widetilde{M}_i}}{\widetilde{M}_i}{M_i}$ we are
going to construct a closed graph orbifold $M_i'$, with $M_i\subset
M_i'$, such that if $\morf{p}{\widetilde{M}_i'}{M_i'}$ is the
universal abelian cover, then
$\morf{p\lvert_{p\inv(M_i)}}{p\inv(M_i)}{M_i}$ is equal to
$\morf{\pi\lvert_{\widetilde{M}_i}}{\widetilde{M}_i}{M_i}$,
i.e., $\widetilde{M}_i=p\inv(M_i)$ and the maps
$p\lvert_{\widetilde{M}_i}$ and $\pi\lvert_{\widetilde{M}_i}$ agree. 

We first look at $M_0'$. We will construct it from $M_0$ by gluing a
solid torus in the boundary of $M_0$, in a way we will now
explain. Now $M'_0$ has splice diagram   

$$\splicediag{8}{30}{
  &&\\
  \Vdots&\overtag\Circ {v_0} {8pt}\lineto[ul]_(.5){n_{01}}
  \lineto[dl]^(.5){n_{0k_0}}
  \lineto[r]^(.25){r_0'}& \Circ\\
  &&\hbox to 0 pt{~.\hss} }$$
where every weight on the left is as in the splice diagram of $M$  if
it does not see the edge we are cutting along. We
want to determine $r_0'$ and the other weights that see the edge we
are cutting along so that the universal abelian cover has the
desired properties.

To determine $\widetilde{M}_0$ we use that the components of a non
connected abelian cover are determined by the map from
$H_1^{orb}$ of the base space 
to the abelian group which determines the non connected cover. 
So in our case  $\widetilde{M}_0$ is determined by
$H_1^{orb}(M_0)\rightarrow H_1^{orb}(M)$.
One makes $M_0'$
by gluing in a solid torus with a orbifold curve of degree $p$ 
such that the generator of
$\ker\big(H_1^{orb}(M_0)\rightarrow H^{orb}_1(M)\big)$ is the curve
that get killed, i.e.\ $\ker\big(H_1^{orb}(M_0)\rightarrow
H^{orb}_1(M)\big)=\ker\big(H_1^{orb}(M_0)\rightarrow
H^{orb}_1(M_0')\big)$, by gluing the primitive element $\alpha$ such that
$\langle p\alpha\rangle=\ker\big(H_1^{orb}(M_0)\rightarrow
H^{orb}_1(M)\big)$ to a meridian 
of the solid torus with a orbifold curve of degree $p$ and a simple
closed curve with intersection $1$ 
with the generator to a longitude. This ensures that $\widetilde{M}_0$
embeds into the universal abelian cover of $M_0'$. We also gets that 
\begin{align*}
H_1(M'_0)&=\im\big(H_1^{orb}(M_0)\rightarrow
H_1^{orb}(M_0')\big)=H_1^{orb}(M_0)/\ker\big(H_1^{orb}(M_0)\rightarrow
H^{orb}_1(M_0')\big)\\
&=H_1^{orb}(M_0)/\ker\big(H_1^{orb}(M_0)\rightarrow
H^{orb}_1(M)\big)=\im\big(H_1^{orb}(M_0)\rightarrow
H_1^{orb}(M)\big).
\end{align*} 
This last fact is what we want to use to find the
splice diagram of $M'_0$, so we need to determine
$\im\big(H_1^{orb}(M_0)\rightarrow H_1^{orb}(M)\big)$.   

We first determine $\ker(H_1^{orb}(M_0)\rightarrow H_1^{orb}(M))$, by
looking at the Meyer Vietoris sequence of the covering of $M$ by
$M_0$ and $M_1$.
\begin{align}\label{meyervie}
\dots\rightarrow H_2^{orb}(M)\rightarrow H_1(T^2)\rightarrow
H_1^{orb}(M_0)\oplus H_1^{orb}(M_1)\rightarrow H_1(M)\rightarrow\dots 
\end{align}
Since we have Poincare duality with rational coefficients, it follows
that $H_2^{orb}(M)$ is finite, so $H_1^2(T^2)=\Z\oplus\Z$ injects into
$H_1^{orb}(M_0)\oplus H_1^{orb}(M_1)$. Hence
$\ker(H_1^{orb}(M_0)\rightarrow H_1^{orb}(M))$ is equal to the
intersection of $H_1^{orb}(M_0)$ with $\Z\oplus\Z$, by using again
that the rational homology of $M_0$ is just as the rational homology
of manifold, it follows that $H_1^{orb}(M_0)$ is rank one. Therefore
$\ker(H_1^{orb}(M_0)\rightarrow H_1^{orb}(M))=\Z$ and is generated by
a class in the boundary of $M_0$. 

Let $Q_0\in H_1(M_0)$ be a representative of the homology class of a
section of the fibration of $T^2$, and let $F_0\in H_1(M_0)$ be a
representative of 
the class of the fibers of the Seifert fibered piece corresponding to
the node $v_0$ in $M_0$. Then some
homology class $T^2$ given by $r_0'Q_0+s_0F_0$ is the class that
gets killed when we glue in $M_1$, so it represents the generator of
$\ker(H_1^{orb}(M_0)\rightarrow H_1^{orb}(M))$.

 We have that
$\num{H_1^{orb}(M_1)/\langle F_0\rangle}=r_0$ by the definition of splice
diagram, since $H_1^{orb}(M_1)/\langle F_0\rangle= H_1^{orb}(M_1/F_0)$. Let 
$\num{H_1^{orb}(M_1)/\langle F_0,Q_0\rangle}=d_1$. Then, since $H_1^{orb}(M_1)/\langle
F_0,Q_0\rangle= H_1^{orb}(M_1/\partial)$, $d_1$ is equal to the ideal
generator defined by Neumann and Wahl in \cite{neumannandwahl2}, which
is an invariant of the splice diagram. This is
their theorem 12.9, whose proof also
works for graph orbifolds. This implies that the order of
$Q_0$ in $H_1^{orb}(M_1)/\langle F_0\rangle$ is $r_0/d_1$. Since
$r_0'Q_0+s_0F_0=0$ in $H_1^{orb}(M_1)$ we get that $r_0'Q_0=0$ in
$H_1^{orb}(M_1)/\langle F_0\rangle$, so $r_0'\mid r_0/d_1$.

We also have the following map of exact sequences 
\begin{align*}
\xymatrix@=5mm{
0\ar[r]&\Z\langle
F_0\rangle\ar[r]&H_1^{orb}(M_1)\ar[r]& H_1^{orb}(M_1)/\langle
F_0\rangle\ar[r] &0 \\
0\ar[r]&\Z\langle
F_0\rangle\ar[r]\ar[u]^\cong &(\Z\times\Z)/\langle
r_0'Q_0+s_0F_0\rangle\ar[r]\ar[u]&
\Z/(r_0')\ar[r]\ar[u] &0 
}
\end{align*}
since the left map is an isomorphism and middle map is injective, it
follows that the right map is injective too, hence
$r'_0=r_0/d_1$. Also note that
$H_1^{orb}(M_1,\partial)=H_1^{orb}(M,M_0)=H_1^{orb}(M)/\im\big(H_1^{orb}(M_0)\to
H_1^{orb}(M)\big)$, so by taking the order of the groups, we get that
$d_1=d/\num{\im\big(H_1^{orb}(M_0)\to H_1^{orb}(M)\big)}$, and since
$H_1^{orb}(M_0')=\im\big(H_1^{orb}(M_0)\to H_1^{orb}(M)\big)$, one
gets that $\num{H_1^{orb}(M_0')}=d/d_1$.

Now for the other edge weights which see the edge we are cutting
along, we start by determining the edge weight on an edge which connects to
our chosen node. Since the fiber intersection number corresponding to
the edge is the same in $M$ and $M_0'$, we get by \ref{nedgedeterminanteq}
the following equations
\begin{align}
\frac{\num{D}}{|H_1^{orb}(M)|}=\frac{\num{D'}}{|H_1^{orb}(M'_0)|}
\end{align}
where $D$ is the edge determinant corresponding to the edge in $M$ and
$D'$ the edge determinant corresponding to the edge in $M'$. Since
$\num{H_1^{orb}(M'_0)}=\num{H_1^{orb}(M)}/d_1$ by the above
calculation, we get that $|D'|=|D|/d_1$. Since
$r_0'=r_0/d_1$, the definition of edge determinants now gives that the
changed edge weight on the edge we are looking at has also been
divided by $d_1$. Continuing
inductively we see that all edge weight that sees the edge are divided
by $d_1$.

 Note that if the curve we kill when we glue in
the solid torus in 
$M_0$ is a multiple of a primitive element, then we get a orbifold
curve in the glued in torus. Now this is not a problem if the weight
on the edge is non zero, so let us consider the case when the edge
weight $r_0=0$. The curve we kill is the curve that bounds in $M_1$,
but $r_0=0$ implies that a fiber in $\partial M_0$ bounds in $M_1$, so
the curve we kill is a fiber. But fibers are primitive elements, so we do not get a
orbifold curve in this case. This insures that $M_1'$ satisfies the
hypothesis of the theorem. 

By a similar argument $r_1'=r_1/d_0$ where
$d_0=\num{H_1^{orb}(M_0,\partial)}=\num{H_1(M,M_1)}$. 

All this implies that the splice diagram $\Gamma_0$ of $M_0'$ look likes 

$$\splicediag{8}{30}{
  &&\\
  \Vdots&\overtag\Circ {v_0} {8pt}\lineto[ul]_(.5){n_{01}}
  \lineto[dl]^(.5){n_{0k_0}}
  \lineto[r]^(.5){r_0/d_1}& \Circ\\
  &&\hbox to 0 pt{~.\hss} }$$
 and the splice diagram $\Gamma_1$ for $M_1'$ is
$$\splicediag{8}{30}{
  &&\\
  \Circ \lineto[r]^(.5){r_1/d_0}& \overtag\Circ{v_1}{8pt}
  \lineto[ur]^(.5){n_{11}}
  \lineto[dr]_(.5){n_{1k_1}}&\Vdots\\
  &&\hbox to 0 pt{~.\hss} }$$
Where all edge weights which see the  cut
edge are gotten from the old
weight by dividing by $d_1$ in the first case and $d_0$ in the second case.

Since $d_0$ and $d_1$ are the ideal generators defined in 12.8 of
\cite{neumannandwahl2}, they are completely determined by $\Gamma$,
since they are generators of ideals defined by numbers coming from
the decorations of $\Gamma$. Now the splice diagrams $\Gamma_0$ and
$\Gamma_1$ have at most $n-1$ nodes each, since we get them by removing at
least one node of $\Gamma$, hence by induction $\widetilde{M}_0'$ is
determined by $\Gamma_0$ and $\widetilde{M}_1'$ are
determined by $\Gamma_1$. Since $\Gamma_0$ and
$\Gamma_1$ are determined by $\Gamma$, this implies that
$\widetilde{M}_0'$ and $\widetilde{M}_1'$ are
determined by $\Gamma$, and therefore by the construction of $M_0'$
and $M_1'$, we have that $\widetilde{M}_0$ and $\widetilde{M}_1$ is
determined by $\Gamma$. 
Next we want to see that the splice diagram determines the number of components of $\pi\inv(M_0)$
and  $\pi\inv(M_1)$ and which components are glued to which. The group of permutations  of the
components of $\pi\inv(M_i)$ is given by 
$H_1^{orb}(M)/\im(H_1^{orb}(M_i)\to H_1^{orb}(M))$. But $H_1^{orb}(M)/\im(H_1^{orb}(M_i)\to
H_1^{orb}(M))= H_1^{orb}(M,M_i)$, so the number of components of $\pi\inv(M_i)$
is the order of $H_1^{orb}(M,M_i)$, which we have seen before is $d_i$, and
only depends on the splice diagram.

To see which components are glued together, we notice that all these
gluings are along tori in $\pi\inv(T^2)$, so the gluing are specified
by the group of permutations of the components of $\pi\inv(T^2)$. So
we look at $H_1^{orb}(M)/\im(H_1^{orb}(T^2)\to H_1(M))$, which is the same as
$H_1^{orb}(M,T^2)$. By excision this is $H_1^{orb}(M/T^2)$. Now
$M/T^2=M_1/T^2\vee M_2/T^2$, so the group of permutations of
$\pi\inv(T^2)$, is given by $H_1^{orb}(M_1,T^2)\oplus H_1^{orb}(M_2,T^2)$.     
Since the group of permutations of the tori is just the product of the
permutations of the group of permutations of the components of each
sides, it follows that a component on the one side is glued to each
component on the other side exactly once. 

The last last thing to see is that the gluing of a component of
$\pi\inv(M_0)$ to a component of $\pi\inv(M_1)$ is specified by the
splice diagram. Let $\widetilde{M}_0$ and $\widetilde{M}_1$ be
components on each of the sides.  To specify the gluing, we show that the splice
diagram determines two distinct essential curves up to homotopy in
each component $T^2_{0i}$ of $\partial \widetilde{M}_0$ and in each
components $T^2_{1j}$ of $\partial \widetilde{M}_1$. The curves are
the same in each of the $T^2_{0i}$'s and likewise in each of the
$T^2_{1i}$'s. If the curves in $T^2_{0i}$ are called $F_0,C_0$ and in
$T^2_{1j}$ the curves are $F_1,C_1$, then when a $T^2_{0i}$ is glued
to a $T^2_{1j}$, $F_0$ is identified with $C_1$ and $C_0$ is
identified with $F_1$.

Now $\widetilde{M}_0$ and $\widetilde{M_1}$ are the
total spaces of graph
orbifolds with boundary completely determined by $\Gamma$, hence we
have naturally defined fibers $\widetilde{F}_i$ by the restriction of
$\pi\inv(F_i)$, where $F_i$ are fibers in the boundary of $M_i$. So
these fibers are going to be one of our curves on each side. The other
curve in $T^2_{ij}$ is then the curve the fiber from the other side are
going to be identified with. So we need to find this curve.

Let $\widetilde{N}_i\subset\widetilde{M}_i$ be the Seifert fibered piece sitting
over the $S^1$ fibered piece of $N_i\subset M_i$ corresponding to the node
$v_i$. Now $\widetilde{N}_i$ is a piece of the JSJ decomposition of
$\widetilde{M}$ and therefore has a rational euler number
$\widetilde{e}_i$. This is the rational euler number of the closed
Seifert fibered manifold one gets by gluing solid tori in $\partial
\widetilde{N}_i$, specified by the gluings of $\widetilde{N}_i$ to
the other pieces of the JSJ decomposition as described in beginning of
Section \ref{det}. We are going compute $\widetilde{e}_i$, and then
use this to specify the other curve in each $T^2_{ij}$. To do this
we can assume by induction, that a simple closed curve transverse to
the fibration is determined in all the boundary
component of $\widetilde{N}_i$ except the boundary components
lying over the edge between $v_0$ and $v_1$ are specified. But the
fibers and the simple closed curves in all the boundary components
lying over the 
edge between $v_0$ and $v_1$, will be the same. Since we know
$\widetilde{F}_i$ the rational euler number $\widetilde{e}_i$ of
$\widetilde{N}_i$ determines a simple closed curve transverse 
to the fibration in each of the boundary components of $\widetilde{M}_i$.

To compute $\widetilde{e}_0$ we will use the relation between
$\widetilde{e}_0$ and the
rational euler number of the $S^1$ fibered piece in the base $N_0$, i.e.,
the $e_{v_0}$. This relation is given by theorem 3.3 of
\cite{brandies}, and give in our situation that
$\widetilde{e}_0=\frac{b_0}{f_0}e_{v_0}$, where $b_0$ is the degree of
the covering 
map restricted to the base of the Seifert fibered pieces and $f_0$ is
the degree of $\pi$ restricted to the fibers. Now
$\deg(\pi\lvert_{\widetilde{M}_0})=b_0f_0$ and
$\deg(\pi\lvert_{\widetilde{M}_0})=\frac{d}{d_0}$ since $d_0$ is the
number of component of $\pi\inv(M_0)$. This implies that
$\widetilde{e}_0=\frac{d}{d_0f_0^2}e_{v_0}$. 

To calculate $f_0$ notice that $f_0=\num{\im(H_1^{orb}(F_0)\to
  H_1^{orb}(M))}$, and since 
\begin{align*}
 \num{H_1^{orb}(M)/\im(H_1^{orb}(F_0)\to
   H_1^{orb}(M))}=\num{H_1^{orb}(M)}/\num{\im(H_1^{orb}(F_0)\to  H_1^{orb}(M))}
\end{align*}
 we get that $f_0=\num{H_1^{orb}(M)}/ \num{H_1^{orb}(M)/\im(H_1^{orb}(F_0)\to
   H_1^{orb}(M))}$, so we want to calculate
 $\num{H_1^{orb}(M)/\im(H_1^{orb}(F_0)\to H_1^{orb}(M))}$.  


Now 
$H_1^{orb}(M)/\im(H_1^{orb}(F_0)\to H_1^{orb}(M))$ is is the same as
$H_1^{orb}(M,F_0)=H_1^{orb}(M/F_0)$, so we need to find
$\num{H_1^{orb}(M/F_0)}$. $M/F_0$ 
  is $M$ with the fibers collapsed at the piece $N_0$; this is the
  same as gluing in  disks in the fibers of
  $N_0$. This then implies that we glue a disk to a simple
  closed curve transverse to the fibration of all the 
  $S^1$--fibered pieces glued to $N_0$ along a torus. This
  implies that $H_1^{orb}(M/F_0)=\big( A_{1}\times A_{2}\times \cdots
  A_{k_0}\times G\big)/\ (a_{1},a_{2},\dots,a_{k_0},g)$,
  where the group $A_{j}$ 
  is the first orbifold homology group of the graph orbifold one gets by
  cutting along a torus corresponding to the $0j$'th edge between the $N_0$
  and the $j$'th piece and gluing in a solid torus. In particular
  $\num{A_{j}}=n_{0j}$. The group $G$ is the first orbifold homology group one
  gets by cutting along the edge between $v_0$ and $v_1$ and gluing in a
  solid torus in the part not containing $v_0$. Hence $\num{G}=r_0$. 
  The $a_{j}\in A_{j}$ are the elements that corresponds to the
  singular fibers over the disk we just glued in; in particular
  $A_{j}/\langle a_{j}\rangle$ is the first orbifold homology group of $M$
  with everything except the part across the edge
  $0j$ collapsed.  Thus $\num{A_{j}/\langle
    a_{j}\rangle}$ is the 
  ideal generator $d_{0j}$ corresponding to the edge $0j$. The same holds for
  $g$ and $G$, especially that $\num{G/\langle g\rangle}=d_{1}$.

 If none of the $A_{j}$ and $G$ is infinite, i.e.,  $n_{0j},r_0\neq
 0$, then
 $\num{H_1^{orb}(M/F_0)}=(r_0\prod_{j=0}^{k_0}n_{0j})/\lcm(n_{01}/d_{01},
 \dots,n_{0k_0}/d_{0k_0},r_0/d_{1})$. 
 Assume that $A_{l}$ is infinite, then all the other groups are
 finite, and we have the following exact sequence
 \begin{align*}
0\rightarrow G\times \prod_{\substack{j=1\\ j\neq l}}^{k_0}A_{j} \rightarrow
\big( A_{1}\times \cdots A_{k_0}\times G\big)/\
(a_{1},\dots,a_{k_0},g)\rightarrow A_{l}/\langle a_{l}\rangle
\rightarrow 0 
\end{align*}
where the $\big( A_{1}\times \cdots A_{k_0}\times G\big)/\
(a_{1},\dots,a_{k},g)\rightarrow A_{l}/\langle
a_{l}\rangle$  map is projection, and $G\prod_{\substack{j=1\\ j\neq
    l}}^{k_0}A_{j}$ is the kernel of this map. This implies that in
this case $\num{H_1^{orb}(M/F_0)}=d_{0l}r_0 \prod_{\substack{j=1\\
    j\neq l}}^{k_0}n_{0j}$. If $G$ is infinite we in the same way get
that that $\num{H_1^{orb}(M/F_0)}=d_{1} \prod_{j=1}^{k_0}n_{0j}$. In
all cases we see that $\num{H_1^{orb}(M/F_0)}=\lambda_0$ only depends
on the splice diagram, since the ideal generators only depends on the splice
diagram. 
 

So now we get that $f_0=\num{H_1^{orb}(M)}/ \num{H_1^{orb}(M/
  F_0)}=\frac{d}{\lambda_0}$, hence
$\widetilde{e}_0=\frac{\lambda_0}{d_0d}e_{v_0}$. 
But the value of $e_{v_0}$ is given by proposition \ref{ceulernumber}, and 
the formula given there shows that $e_{v_0}$ is $d$ times a number
given by the splice diagram, hence $\frac{\lambda_0}{d_0d}e_{v_0}$ is
$\frac{\lambda_0}{d_0}$ times a number only depending on $\Gamma$, hence
$\widetilde{e}_0$ only depends on $\Gamma$. We can in the same way
calculate $\widetilde{e}_1$ and see that it also only depends on $\Gamma$.
This implies that the splice diagram
specifies a simple closed curve $C_i$ transverse to the fibration of the
$T^2_{ij}$'s, which in particular is not null homologous.

Now the gluing of $\widetilde{M}_0$ to $\widetilde{M}_1$ is specified
by identifying $\widetilde{F}_0$ with $C_1$ and $\widetilde{F}_1$ with
$C_0$. But since the $\widetilde{F}_i$'s and the $C_i$'s are determined
by $\Gamma$, the gluing is determined by $\Gamma$, and hence
$\widetilde{M}$ is determined by $\Gamma$. 
\end{proof}

\begin{cor}
Let $M$ be a rational homology sphere graph manifold with splice
diagram $\Gamma(M)$, such that around any node in $\Gamma(M)$ the
edge weights are pairwise coprime, then the universal abelian cover of
$M$ is an integer homology sphere.
\end{cor}

\begin{proof}
It is shown in \cite{EisenbudNeumann} that a splice diagram with
pairwise coprime edge weights at nodes is the
splice diagram for an integer homology sphere.
 So given a splice
diagram satisfying the assumption there is a integer homology sphere
$M'$ with splice diagram $\Gamma(M)$, hence $M'$ is the universal
abelian cover of $M$ by \ref{universalabcover}. 
\end{proof}   

This actually gives a way to construct the universal abelian cover for any
rational homology sphere graph manifold  that
has a splice diagram with 
pairwise coprime edge weights at nodes, since \cite{EisenbudNeumann}
describes how to construct the integral homology sphere by splicing,
and to construct a plumbing diagram for it in the case where we have
positive decorations at nodes. If we want a construction of the
plumbing diagram in the case with
negative decorations at nodes it follows from theorem 3.1 in
\cite{bilinearforms} which give us that given a splice diagram
$\Gamma(M)$ as above, one can construct an unimodular tree
$\Delta(M)$, which will be the plumbing diagram.

\newpage

\bibliography{splicbibliografi}

\end{document}